\documentclass[11pt,letterpaper, reqno]{amsart}

\newcommand{\la}{\langle}
\newcommand{\ra}{\rangle}
\renewcommand{\Re}{\operatorname{Re}}
\renewcommand{\Im}{\operatorname{Im}}

\newcommand{\supp}{\operatorname{supp}}

\usepackage{amssymb}
\usepackage{amsthm}
\usepackage{amsxtra}
\usepackage{graphicx}
\usepackage{mathabx}

\newtheorem{theorem}{Theorem}

\newtheorem{proposition}[theorem]{Proposition}
\newtheorem{lemma}[theorem]{Lemma}
\newtheorem{corollary}[theorem]{Corollary}

\theoremstyle{remark}
\newtheorem{remark}[theorem]{Remark}

\numberwithin{equation}{section}

\numberwithin{theorem}{section}

\numberwithin{table}{section}

\numberwithin{figure}{section}

\ifx\pdfoutput\undefined
  \DeclareGraphicsExtensions{.pstex, .eps}
\else
  \ifx\pdfoutput\relax
    \DeclareGraphicsExtensions{.pstex, .eps}
  \else
    \ifnum\pdfoutput>0
      \DeclareGraphicsExtensions{.pdf}
    \else
      \DeclareGraphicsExtensions{.pstex, .eps}
    \fi
  \fi
\fi

\title{A Spectral Multiplier Theorem associated with a Schr\"odinger Operator}

\date{\today}
\author{Younghun Hong}
\address{University of Texas at Austin}
\email{yhong@math.utexas.edu}

\begin{document}

\maketitle

\begin{abstract}
We establish a H\"ormander type spectral multiplier theorem for a Schr\"odinger operator $H=-\Delta+V(x)$ in $\mathbb{R}^3$, provided $V$ is contained in a large class of short range potentials. This result does not require the Gaussian heat kernel estimate for the semigroup $e^{-tH}$, and indeed the operator $H$ may have negative eigenvalues. As an application, we show local well-posedness of a 3d quintic nonlinear Schr\"odinger equation with a potential.
\end{abstract}



\section{Introduction}

\subsection{Statement of the main theorem}
In this paper, we establish a H\"ormander type spectral multiplier theorem for a Schr\"odinger operator $H=-\Delta+V$ in $\mathbb{R}^3$, provided that $V$ is contained in a large class of short range potentials. Precisely, we assume that $V$ is contained in $\mathcal{K}_0\cap L^{3/2,\infty}$, where $\mathcal{K}_0$ is the norm closure of bounded, compactly supported functions with respect to the \textit{global Kato norm}
\begin{equation}\label{global Kato norm}
\|V\|_{\mathcal{K}}:=\sup_{x\in\mathbb{R}^3} \int_{\mathbb{R}^3}\frac{|V(y)|}{|x-y|}dy,
\end{equation}
and $L^{3/2,\infty}$ is the weak $L^{3/2}$-space. We also assume that $H$ has no eigenvalue or resonance on the positive real-line $[0,+\infty)$. By a \textit{resonance}, we mean a complex number $\lambda$ such that the equation $\psi+(-\Delta-\lambda\pm i0)^{-1}V\psi=0$ has a slowly decaying solution $\psi\in L^{2,-s}\setminus L^2$ for any $s>\frac{1}{2}$, where $L^{2,s}=\{\la x\ra^s f\in L^2\}$.

By the above assumptions, the operator $H$ is self-adjoint on $L^2$. Moreover, its spectrum $\sigma(H)$ consists of purely absolutely continuous spectrum on the positive real-line $[0,+\infty)$ and at most finitely many negative eigenvalues \cite{BG}. Therefore, for a bounded Borel function $m:\sigma(H)\subset\mathbb{R}\to\mathbb{C}$, one can define a spectral multiplier $m(H)$ as a bounded operator on $L^2$ via functional calculus. 

A natural question is then to find a sufficient condition to extend boundedness of the multiplier $m(H)$ to $L^p$ for $p\neq 2$. Such a condition is typically given in terms of regularity of symbols. To measure regularity of a symbol $m:\sigma(H)\to\mathbb{C}$, we define a Sobolev type norm by
\begin{equation}\label{Hs}
\|m\|_{\mathcal{H}(s)}:=\sum_{\lambda_j:\textup{ negative eigenvalues}} |m(\lambda_j)|+\sup_{t>0}\|\chi(\lambda)m((t\lambda)^2)\|_{W_\lambda^{s,2}((0,+\infty))},
\end{equation}
where $\chi\in C_c^\infty(\mathbb{R})$ is a standard dyadic partition of unity function such that $\chi$ is supported in $[\frac{1}{2},2]$ and $\sum_{N\in 2^{\mathbb{Z}}}\chi(\tfrac{\cdot}{N})\equiv 1$ on $(0,+\infty)$, and $W^{s,2}$ is  the $L^2$-Sobolev space of order $s$.

Our main result is the following.

\begin{theorem}[Spectral multiplier theorem]\label{main theorem}
Suppose that $V\in \mathcal{K}_0\cap L^{3/2,\infty}$ and $H=-\Delta+V$ has no eigenvalue or resonance on $[0,+\infty)$. We also assume that for $s>2$, the symbol $m:\sigma(H)\to\mathbb{C}$ satisfies $\|m\|_{\mathcal{H}(s)}<\infty$. Then, we have
\begin{equation}
\|m(H)\|_{L^p\to L^p}\lesssim \|m\|_{\mathcal{H}(s)},\quad\forall 1<p<\infty.
\end{equation}
\end{theorem}

When $V=0$, Theorem \ref{main theorem} is simply the classical H\"ormander-Mikhlin multiplier theorem \cite{C}.

There are several ways to prove the spectral multiplier theorem for Schr\"odinger operators. For an operator $A$, we say that the semigroup $e^{-tA}$ satisfies the Gaussian heat kernel estimate if the kernel of $e^{-tA}$, denoted by $e^{-tA}(x,y)$, obeys
\begin{equation}\label{Gaussian heat kernel}
e^{-tA}(x,y)\lesssim t^{-3/2}e^{-\frac{|x-y|^2}{ct}},\quad\forall t>0
\end{equation}
for some $c>0$. Gaussian upper bounds for the heat kernels have been used successfully to prove spectral multiplier theorems for rather general operators, not necessarily Schr\"odinger operators (see \cite{C, MM, DOS} and references therein). In the case of the Schr\"odinger operator $H=-\Delta+V$ in $\mathbb{R}^3$, if $V_+=\max(V,0)$ is in local Kato class, that is,
\begin{equation}
\lim_{r\to0+}\sup_{x\in\mathbb{R}^3}\int_{|x-y|\leq r}\frac{|V_+(y)|}{|x-y|}dy=0,
\end{equation}
and if $V_-=\min(V,0)\in\mathcal{K}_0$ and $\|V_-\|_{\mathcal{K}}<4\pi$, then it is known that the semigroup $e^{-tH}$ satisfies the Gaussian heat kernel estimate \eqref{Gaussian heat kernel} \cite{T, DP}. The spectral multiplier theorem for $H$ then follows from \cite[Theorem 3.1]{DOS}. However, for Gaussian upper bounds \eqref{Gaussian heat kernel}, operators need to be positive definite, while the Schr\"odinger operator in Theorem \ref{main theorem} may have negative eigenvalues.

One can also use the wave operators to show the spectral multiplier theorem. The forward-in-time (backward-in-time, resp) wave operator of the Schr\"odinger operator $H=-\Delta+V$ is defined by
\begin{equation}
W_+=\underset{t\to+\infty}{s\mbox{-}\lim}e^{itH}e^{-it(-\Delta)}\quad\Big(W_-=\underset{t\to-\infty}{s\mbox{-}\lim}e^{itH}e^{-it(-\Delta)}\textup{, resp}\Big).
\end{equation}
An important feature of wave operators is its intertwining property, that is, $P_cf(H)=W_\pm f(-\Delta)(W_\pm)^*$, where $P_c$ is the spectral projection to the continuous spectrum and $(W_\pm)^*$ is the dual of $W_\pm$. In \cite{Y}, Yajima proved that the wave operators $W_\pm$ are bounded on $L^p$ for all $1\leq p\leq\infty$, provided that $|V(x)|\lesssim\la x\ra^{-5-\epsilon}$ for $\epsilon>0$, and zero is not an eigenvalue or a resonance of $H$. Later, in \cite{B}, Beceanu extended this result to a larger space
\begin{equation}
B:=\Big\{V: \sum_{k=-\infty}^\infty 2^{k/2}\|V(x)\|_{L_x^2(2^k\leq|x|<2^{k+1})}<\infty\Big\}.
\end{equation}
The spectral multiplier theorem then follows immediately from the intertwining property and boundedness of wave operators and the classical H\"ormander-Mikhlin multiplier theorem, since
\begin{equation}
\begin{aligned}
\|P_cf(H)\|_{L^p\to L^p}&=\|W_\pm f(-\Delta)(W_\pm)^*\|_{L^p\to L^p}\\
&\lesssim \|f(-\Delta)(W_\pm)^*\|_{L^p\to L^p}\lesssim \|(W_\pm)^*\|_{L^p\to L^p}<\infty
\end{aligned}
\end{equation}
and $(I-P_c)f(H)$ is bounded on $L^p$ by Lemma \ref{eigenfunctions}. Theorem \ref{main theorem} improves the spectral multiplier theorem as a consequence boundedness of the wave operator, in that the potential class $\mathcal{K}_0\cap L^{3/2,\infty}$ is larger than the potential class $B$. Note that a potential having many singular points, such as $\sum_{k=1}^N 1_{|x-x_j|\leq 1}\frac{1}{|x-x_j|^{2-\epsilon}}$ with $x_j\neq x_k$ and $\epsilon>0$, is contained in $\mathcal{K}_0\cap L^{3/2,\infty}$, but not in $B$.

Our proof of the spectral multiplier theorem is perturbative, and it relies heavily on the explicit integral representation of the kernel of the multiplier. We consider the spectral multiplier $m(H)P_c$ as a perturbation of the Fourier multiplier $m(-\Delta)$, and then we show that the difference $(m(H)P_c-m(-\Delta))$ is bounded on $L^p$. In order to estimate the difference, we first decompose it into its dyadic pieces
\begin{equation}
\sum_{N\in 2^{\mathbb{Z}}}\chi(\tfrac{\sqrt{H}}{N})\Big(m(H)-m(-\Delta)\Big),
\end{equation}
where $\chi$ is the function given in \eqref{Hs}. Then, we generate a formal series expansion for each dyadic piece to get explicit integral representations of kernels of terms in the series using the free resolvent formula 
\begin{equation}\label{free resolvent formula}
((-\Delta-z)^{-1}f)(x)=\int_{\mathbb{R}^3}\frac{e^{i\sqrt{z}|x-y|}}{4\pi|x-y|} f(y)dy.
\end{equation}
We estimate these integral kernels. Summing them up, we prove the spectral multiplier theorem.

A key observation is that in spite of the singular integral nature of both $m(H)P_c$ and $m(-\Delta)$ as Calderon-Zygmund operators, the kernel of their difference is less singular than usual Calderon-Zygmund operators. This fact is essential in our analysis, since it allows us to avoid using the delicate classical Calderon-Zygmund theory for the complicated operator $m(H)$ (see Remark 4.4). Instead, we just make use of the fractional integration inequality and H\"older inequality.

\subsection{Application to NLS}
The choice of the potential class in the main theorem is motivated by the following nonlinear application.

First, we recall the Strichartz estimates for the linear propagator $e^{-itH}$.
\begin{proposition}[Strichartz estimates]\label{prop:StrichartzEstimate}
If $V\in\mathcal{K}_0$ and $H$ has no eigenvalue or resonance on $[0,+\infty)$, then
\begin{align}
\|e^{-itH}P_cf\|_{L_t^q L_x^r}&\lesssim\|f\|_{L^2},\\
\Big\|\int_0^t e^{-i(t-s)H}P_cF(s)ds\Big\|_{L_t^q L_x^r}&\lesssim\|F\|_{L_t^{2}L_x^{6/5}},
\end{align}
where $\frac{2}{q}+\frac{3}{r}=\frac{3}{2}$ and $2\leq q,r\leq\infty$.
\end{proposition}

\begin{proof} Beceanu-Goldberg \cite{BG} proved the dispersive estimate
\begin{equation}\label{dispersive estimate}
\|e^{-itH}P_c\|_{L^1\to L^\infty}\lesssim|t|^{-3/2},
\end{equation}
where $P_c$ is the spectral projection to the continuous spectrum. Strichartz estimates then follow by the argument of Keel-Tao \cite{KT}.
\end{proof}

\begin{remark}
The dispersive estimate of the form \eqref{dispersive estimate} was first proved by Journ\'e-Soffer-Sogge under suitable assumptions on potentials \cite{JSS} . The assumptions have been relaxed by Rodnianksi-Schlag \cite{RSch}, Goldberg-Schlag \cite{GS1} and Goldberg \cite{G1, G2}. Recently, Beceanu-Goldberg established \eqref{dispersive estimate} for a scaling-critical potential class $\mathcal{K}_0$ \cite{BG}.
\end{remark}

An interesting question is then whether one can use the above Strichartz estimates to show the local well-posedness (LWP), for instance, for a 3d quintic nonlinear Schr\"odinger equation with a potential
\begin{equation}\tag{$(\textup{NLS}_V)$}
iu_t+\Delta u-Vu\pm|u|^4u=0;\ u(0)=u_0
\end{equation}
assuming that $V$ satisfies the conditions in Proposition \ref{prop:StrichartzEstimate}. However, if one tries to show local well-posedness by the standard contraction mapping argument as in \cite{C,Tao1}, one will realize that there is a subtle problem, mainly because the linear propagator $e^{-itH}$ does not commute with the differential operators from the Sobolev norms.

We overcome this subtle problem by the two norm estimates lemma, whose proof relies on the spectral multiplier theorem.
\begin{lemma}[Two norm estimates] If $V\in\mathcal{K}_0\cap L^{3/2,\infty}$ and $H$ has no eigenvalue or resonance on the positive real-line $[0,+\infty)$, then
\begin{align}
\|H^\frac{s}{2}P_c(-\Delta)^{-\frac{s}{2}}f\|_{L^r}&\lesssim\|f\|_{L^r},\\
\|(-\Delta)^{\frac{s}{2}}H^{-\frac{s}{2}}P_cf\|_{L^r}&\lesssim\|f\|_{L^r}.
\end{align}
for $0\leq s\leq 2$ and $1<r<\frac{3}{s}$.
\end{lemma}

Together with Strichartz estimates and the two norm estimates lemma, we prove local well-posedness.

\begin{theorem}[LWP] Suppose that $V\in\mathcal{K}_0\cap L^{3/2,\infty}$ and $H$ has no eigenvalue or resonance on the positive real-line $[0,+\infty)$. Then, $(\textup{NLS}_V)$ is locally well-posed in $\dot{H}^1$.
\end{theorem}

\begin{remark}
$(i)$ The range of $r$ in the two norm estimates lemma is sharp. See the counterexample in \cite{Sh}.\\
$(ii)$ The additional hypothesis $V\in L^{3/2,\infty}$, compared to Strichartz estimates, is from the two norm estimates lemma. In the proof of the two norm estimates lemma, we used this additional assumption. 
\end{remark}

\subsection{Organization of the paper} The outline of the proof of Theorem \ref{main theorem} is given in \S 2. We decompose the spectral representation of the difference $(m(H)P_c-m(-\Delta))$ into the low, medium and high frequencies, and then analyze them separately in \S 4-6. In \S7, we establish LWP of a 3d quintic nonlinear Schr\"odinger equation with a potential.

\subsection{Notations} 
For an integral operator $T$, its integral kernel is denoted by $T(x,y)$. We denote by $A``="B$ the formal identity which  will be proved later. 

\subsection{Acknowledgement} The author would like to thank his advisor, Justin Holmer, for his help and encouragement. He also thank an anonymous referee for very helpful suggestions to improve this article.

\section{Reduction to the Key Lemma}

Suppose that $V\in\mathcal{K}_0$ and $H$ has no eigenvalue or resonance on $[0,+\infty)$. Then, the spectrum of $H$, denoted by $\sigma(H)$, consists of purely continuous spectrum on the positive real-line $[0,+\infty)$ and at most finitely many negative eigenvalues. For $z\notin\sigma(H)$, we define the resolvent by $R_V(z):=(H-z)^{-1}$, and denote
\begin{equation}
R_V^\pm(\lambda):=\underset{{\epsilon\to0+}}{\textup{s-}\lim} R_V(\lambda\pm i\epsilon).
\end{equation}
Let $P_c$ be the spectral projection on the continuous spectrum. Then, by the Stone's formula, the spectral multiplier operator $m(H)P_c$ is represented by
\begin{equation}
m(H)P_c=\frac{1}{2\pi i}\int_0^\infty m(\lambda) [R_V^+(\lambda)-R_V^-(\lambda)]d\lambda=\frac{1}{\pi}\int_0^\infty m(\lambda)\Im R_V^+(\lambda)d\lambda.
\end{equation}
Applying the identity 
\begin{equation}
\begin{aligned}
R_V^+(\lambda)&=R_0^+(\lambda)(I+VR_0^+(\lambda))^{-1}\\
&=R_0^+(\lambda)\Big(I-(I+VR_0^+(\lambda))^{-1}VR_0^+(\lambda)\Big)\\
&=R_0^+(\lambda)-R_0^+(\lambda)(I+VR_0^+(\lambda))^{-1}VR_0^+(\lambda),
\end{aligned}
\end{equation}
we split $m(H)P_c$ into the pure and the perturbed parts,
\begin{equation}
\begin{aligned}
m(H)P_c&=\frac{1}{\pi}\int_0^\infty m(\lambda)\Im R_0^+(\lambda) d\lambda\\
&-\frac{1}{\pi}\int_0^\infty m(\lambda)\Im [R_0^+(\lambda)(I+VR_0^+(\lambda))^{-1}VR_0^+(\lambda)] d\lambda\\
&=:m(-\Delta)+\textup{Pb},
\end{aligned}
\end{equation}
where $m(-\Delta)$ is the Fourier multiplier such that $\widehat{m(-\Delta) f}(\xi)=m(|\xi|^2)\hat{f}(\xi)$. For the pure part $m(-\Delta)$, it follows from the classical H\"ormander-Mikhlin multiplier theorem \cite{Hor} that for $s>\frac{3}{2}$,
\begin{equation}
\|m(-\Delta)\|_{L^p\to L^p}\lesssim \|m\|_{\mathcal{H}(s)},\quad\forall 1<p<\infty.
\end{equation}
Therefore, it suffices to show boundedness of the perturbed part. For the perturbed part $\textup{Pb}$, we further decompose it into dyadic pieces. Let $\chi$ be the smooth dyadic partition of unity function chosen in \eqref{Hs}, and decompose 
\begin{equation}
\textup{Pb}=\sum_{N\in 2^{\mathbb{Z}}}\textup{Pb}_N,
\end{equation}
where
\begin{equation}\label{PbN}
\textup{Pb}_N:=-\frac{1}{\pi}\int_0^\infty m(\lambda)\chi_N(\sqrt{\lambda})\Im[R_0^+(\lambda)(I+VR_0^+(\lambda))^{-1}VR_0^+(\lambda)]d\lambda.
\end{equation}
For a small dyadic number $N_0$ and a large dyadic number $N_1$ to be chosen later, we denote the low (high, resp) frequency part by 
\begin{equation}
\textup{Pb}_{\leq N_0}:=\sum_{N\leq N_0}\textup{Pb}_N\ \Big(\textup{Pb}_{\geq N_1}:=\sum_{N\geq N_1}\textup{Pb}_N\textup{, resp}\Big).
\end{equation}

In the next four sections, we will show the following lemma.
\begin{lemma}[Key lemma]\label{Key} Suppose that $V\in\mathcal{K}_0\cap L^{3/2,\infty}$ and $H$ has no eigenvalue or resonance on $[0,+\infty)$. Let $s>2$. Then, there exists $p>1$ but sufficiently close to $1$ such  that the following hold.\\
$(i)$ (High frequency) There exists $N_1=N_1(V)\gg1$ such that 
\begin{equation}
\|\textup{Pb}_{\geq N_1}\|_{L^{p,1}\to L^{p,\infty}}\lesssim \|m\|_{\mathcal{H}(s)},
\end{equation}
where $L^{p,1}$ and $L^{p,\infty}$ are the Lorentz spaces (see Appendix A).\\
$(ii)$ (Low frequency) There exists $N_0=N_0(V)\ll1$ such that 
\begin{equation}
\|\textup{Pb}_{\leq N_0}\|_{L^{p,1}\to L^{p,\infty}}\lesssim \|m\|_{\mathcal{H}(s)}.
\end{equation}
$(iii)$ (Medium frequency) For $N_0<N<N_1$,
\begin{equation}
\|\textup{Pb}_N\|_{L^{p,1}\to L^{p,\infty}}\lesssim_{N_0,N_1} \|m\|_{\mathcal{H}(s)}.
\end{equation}
\end{lemma}

\begin{proof}[Proof of Theorem \ref{main theorem}, assuming Lemma \ref{Key}]
Let $p>1$ be sufficiently close to $1$ as in Lemma \ref{Key}. Summing the estimates in Lemma \ref{Key}, we prove that $\textup{Pb}$ is bounded from $L^{p,1}$ to $L^{p,\infty}$. Then, it follows from the classical H\"ormander-Mikhlin multiplier theorem that $m(H)P_c=m(-\Delta)+\textup{Pb}$ is bounded from $L^{p,1}$ to $L^{p,\infty}$. Moreover, by Lemma \ref{eigenfunctions} (see below), $m(H): L^{p,1}\to L^{p,\infty}$ is bounded.

Recall that by functional calculus, $m(H)$ is bounded on $L^2$. Thus, by the real interpolation lemma (Corollary \ref{Interpolation}), $m(H)$ is bounded on $L^p$ for all $1<p\leq2$. Finally, applying the spectral multiplier theorem to the symbol $\bar{m}$ and the standard duality argument with $m(H)=\bar{m}(H)^*$, we conclude that $m(H)$ is bounded on $L^p$ for $2<p<\infty$.
\end{proof}

\section{Preliminaries}

\subsection{Resolvent estimates}
Following Beceanu-Goldberg \cite{BG}, we collect kernel estimates for $VR_0^+(\lambda)$,  $V(R_0^+(\lambda)-R_0^+(\lambda_0))$, $(VR_0^+(\lambda))^4$ and $(I+VR_0^+(\lambda))^{-1}$, all of which will play as \textit{building blocks} to analyze the kernel of $\textup{Pb}_N$.

\begin{lemma}[Resolvent estimates]\label{VR estimate} Suppose that $V\in\mathcal{K}_0$.\\
$(i)$ For $\lambda\geq0$,
\begin{equation}
\|VR_0^+(\lambda)f\|_{L^1}\leq \frac{\|V\|_{\mathcal{K}}}{4\pi}\|f\|_{L^1}.
\end{equation}
$(ii)$ Define the difference operator by
\begin{equation}
B_{\lambda,\lambda_0}:=V(R_0^+(\lambda)-R_0^+(\lambda_0)).
\end{equation}
For $\epsilon>0$, there exist $\delta>0$ and an integral operator $B: L^1\to L^1$ such that for $|\lambda-\lambda_0|\leq\delta$ and $\lambda,\lambda_0\geq0$,
\begin{equation}
|B_{\lambda,\lambda_0}(x,y)|\leq B(x,y)\textup{, and }\|B(x,y)\|_{L_y^\infty L_x^1}\leq\epsilon.
\end{equation}
$(iii)$ For $\epsilon>0$, there exist $N_1\gg1$ and an integral operator $D=D_\epsilon: L^1\to L^1$ such that for $\lambda\geq N_1$,
\begin{equation}
|(VR_0^+(\lambda))^4(x,y)|\leq D(x,y)\textup{, and }\|D(x,y)\|_{L_y^\infty L_x^1}\leq\epsilon.
\end{equation}
\end{lemma}

\begin{proof}$(i)$ By the free resolvent formula $R_0^+(\lambda)(x,y)=\frac{e^{i\sqrt{\lambda}|x-y|}}{4\pi|x-y|}$, the Minkowski inequality and the definition of the global Kato norm \eqref{global Kato norm}, we have
\begin{equation}
\|VR_0^+(\lambda)f\|_{L^1}\leq\int_{\mathbb{R}^3} \Big\| \frac{|V(x)|}{4\pi|x-y|}\Big\|_{L_x^1} |f(y)|dy\leq\frac{\|V\|_{\mathcal{K}}}{4\pi}\|f\|_{L^1}.
\end{equation}
$(ii)$ For $\epsilon>0$, decompose $V=V_1+V_2$ such that $V_1$ is bounded and compactly supported and $\|V_2\|_\mathcal{K}\leq\epsilon$. We choose $\delta>0$ such that $|\sqrt{\lambda}-\sqrt{\lambda_0}|\leq\epsilon\|V_1\|_{L^1}^{-1}$ for all $\lambda,\lambda_0\geq0$ with $|\lambda-\lambda_0|\leq\delta$. By the mean-value theorem,
\begin{equation}
\begin{aligned}
|B_{\lambda,\lambda_0}(x,y)|&\leq\Big|\frac{V_1(x)(e^{i\sqrt{\lambda}|x-y|}-e^{i\sqrt{\lambda_0}|x-y|})}{4\pi|x-y|}\Big|+\Big|\frac{V_2(x)(e^{i\sqrt{\lambda}|x-y|}-e^{i\sqrt{\lambda_0}|x-y|})}{4\pi|x-y|}\Big|\\
&\leq\frac{|V_1(x)||\sqrt{\lambda}-\sqrt{\lambda_0}|}{4\pi}+\frac{|V_2(x)|}{2\pi|x-y|}\\
&\leq\frac{\epsilon|V_1(x)|}{4\pi\|V_1\|_{L^1}}+\frac{|V_2(x)|}{2\pi|x-y|}=:B_\epsilon(x,y).
\end{aligned}
\end{equation}
Then, we have
\begin{equation}
\|B_\epsilon(x,y)\|_{L_y^\infty L_x^1}\leq\frac{\epsilon}{4\pi}+\frac{\|V_2\|_{\mathcal{K}}}{2\pi}\leq\epsilon.
\end{equation}
$(iii)$ Similarly, for $\epsilon>0$, decompose $V=V_1+V_2$ such that $V_1$ is bounded and compactly supported and $\|V_2\|_\mathcal{K}\leq\epsilon\|V\|_{\mathcal{K}}^{-3}$. We then write 
\begin{equation}
|(VR_0^+(\lambda))^4(x,y)|\leq|(V_1 R_0^+(\lambda))^4(x,y)|+|(VR_0^+(\lambda))^4(x,y)-(V_1 R_0^+(\lambda))^4(x,y)|.
\end{equation}
For the first term, by the fractional integration inequalities, the H\"older inequalities in the Lorentz spaces (Lemma A.2) and the free resolvent estimate $\|R_0^+(\lambda)\|_{L^{4/3}\to L^4}\lesssim\la\lambda\ra^{-1/4}$ \cite[Lemma 2.1]{GS2}, we get
\begin{equation}
\begin{aligned}
&\|R_0^+(\lambda)(V_1 R_0^+(\lambda))^3f\|_{L^\infty}\\
&\lesssim \|(V_1 R_0^+(\lambda))^3f\|_{L^{3/2,1}}\leq\|V_1\|_{L^{3,1}}\|R_0^+(\lambda)(V_1 R_0^+(\lambda))^2f\|_{L^{3,\infty}}\\
&\lesssim\|(V_1 R_0^+(\lambda))^2f\|_{L^1}\leq\|V_1\|_{L^{4/3}} \|R_0^+(\lambda) V_1 R_0^+(\lambda) f\|_{L^4}\\
&\lesssim\la\lambda\ra^{-1/4}\|V_1 R_0^+(\lambda)f\|_{L^{4/3}}\lesssim\la\lambda\ra^{-\frac{1}{4}}\|V_1\|_{L^\infty}\|R_0^+(\lambda)f\|_{L_{x\in\supp V_1}^{4/3}}\\&\lesssim\la\lambda\ra^{-\frac{1}{4}}\int_{\mathbb{R}^3}\Big\|\frac{1}{|x-y|}\Big\|_{L_{x\in\supp V_1}^{4/3}}|f(y)| dy\lesssim\la\lambda\ra^{-\frac{1}{4}}\|f\|_{L^1}.
\end{aligned}
\end{equation}
Taking $f\to\delta(\cdot-y)$, we obtain that $|R_0^+(\lambda)(V_1 R_0^+(\lambda))^3(x,y)|\to 0$ as $\lambda\to+\infty$. Thus, there exists $N_1=N_1(\epsilon, V_1)\gg1$ such that if $\lambda\geq N_1$, then 
\begin{equation}
|(V_1 R_0^+(\lambda))^4)(x,y)|\leq\frac{\epsilon|V_1(x)|}{2\|V_1\|_{L^1}}=:D_1(x,y).
\end{equation}
Then, it is obvious that $\|D_1(x,y)\|_{L_{y}^\infty L_x^1}\leq\frac{\epsilon}{2}$. For the second term, we split 
\begin{equation}
\begin{aligned}
&(VR_0^+(\lambda))^4(x,y)-(V_1 R_0^+(\lambda))^4(x,y)\\
&=(V_2R_0^+(\lambda)(VR_0^+(\lambda))^3)(x,y)+(V_1 R_0^+(\lambda)V_2R_0^+(\lambda)(VR_0^+(\lambda))^2)(x,y)\\
&+((V_1 R_0^+(\lambda))^2V_2R_0^+(\lambda)VR_0^+(\lambda))(x,y)+((V_1 R_0^+(\lambda))^3V_2R_0^+(\lambda))(x,y).
\end{aligned}
\end{equation}
Since the kernel of $R_0^+(\lambda)$ is bounded by the kernel of $(-\Delta)^{-1}$, we have
\begin{equation}
\begin{aligned}
&|(VR_0^+(\lambda))^4(x,y)-(V_1 R_0^+(\lambda))^4(x,y)|\\
&\leq(|V_2|(-\Delta)^{-1}(|V|(-\Delta)^{-1})^3)(x,y)\\
&+(|V_1| (-\Delta)^{-1}|V_2|(-\Delta)^{-1}(|V|(-\Delta)^{-1})^2)(x,y)\\
&+((|V_1| (-\Delta)^{-1})^2|V_2|(-\Delta)^{-1}|V|(-\Delta)^{-1})(x,y)\\
&+((|V_1| (-\Delta)^{-1})^3|V_2|(-\Delta)^{-1})(x,y)\\
&=:D_2(x,y).
\end{aligned}
\end{equation}
Then, 
\begin{equation}
\begin{aligned}
\|D_2(x,y)\|_{L_y^\infty L_x^1}&=\|D_2\|_{L^1 \to L^1}\\
&\leq\||V_2|(-\Delta)^{-1}\|_{L^1\to L^1}\||V|(-\Delta)^{-1}\|_{L^1\to L^1}^3\\
&+\||V_1|(-\Delta)^{-1}\|_{L^1\to L^1}\||V_2|(-\Delta)^{-1}\|_{L^1\to L^1}\||V|(-\Delta)^{-1}\|_{L^1\to L^1}^2\\
&+\||V_1|(-\Delta)^{-1}\|_{L^1\to L^1}^2\||V_2|(-\Delta)^{-1}\|_{L^1\to L^1}\||V|(-\Delta)^{-1}\|_{L^1\to L^1}\\
&+\||V_1|(-\Delta)^{-1}\|_{L^1\to L^1}^3\||V_2|(-\Delta)^{-1}\|_{L^1\to L^1}\\
&\leq 4 \Big(\frac{\|V\|_{\mathcal{K}}+\|V_2\|_{\mathcal{K}}}{4\pi}\Big)^3 \frac{\|V_2\|_{\mathcal{K}}}{4\pi}\leq\frac{\epsilon}{2},
\end{aligned}
\end{equation}
where $D_2$ is an integral operator with kernel $D_2(x,y)$. Therefore, we conclude that 
\begin{equation}
|(VR_0^+(\lambda))^4(x,y)|\leq D(x,y):=D_1(x,y)+D_2(x,y)
\end{equation}
and $\|D(x,y)\|_{L_y^\infty L_x^1}\leq\epsilon$.
\end{proof}

By algebra, the resolvent $R_V^+(\lambda)$ can be written as
\begin{equation}
R_V^+(\lambda)``="R_0^+(\lambda)(I+VR_0^+(\lambda))^{-1}.
\end{equation}
Let $\mathcal{L}(L^1)$ be the space of bounded operators on $L^1$. The following lemmas say that $(I+VR_0^+(\lambda))$ is invertible in $\mathcal{L}(L^1)$ for $\lambda\geq0$, its inverse $(I+VR_0^+(\lambda))^{-1}$ is uniformly bounded in $\mathcal{L}(L^1)$, and is the sum of the identity map and an integral operator.

\begin{lemma}[Invertibility of $(I+VR_0^+(\lambda))$] 
If $V\in\mathcal{K}_0$ and $H$ has no eigenvalue or resonance on $[0,+\infty)$, then $(I+VR_0^+(\lambda))$ is invertible in $\mathcal{L}(L^1)$ for $\lambda\geq0$. 
\end{lemma}
\begin{proof}
If it is not invertible, there exists $\varphi\in L^1$, $\varphi\neq0$, such that $(I+VR_0^+(\lambda))\varphi=0$. Then, $\psi:=R_0^+(\lambda)\varphi$ solves the eigenvalue equation $(-\Delta+V)\psi=(\lambda+i0)\psi\Longleftrightarrow\psi+R_0^+(\lambda)V\psi=0$. Moreover, by the resolvent formula $R_0^+(\lambda)(x,y)=\frac{e^{i\sqrt{\lambda}|x-y|}}{4\pi|x-y|}$, if $s>\frac{1}{2}$, then
\begin{equation}
\begin{aligned}
\|\la x\ra^{-s}\psi\|_{L^2}&=\|\la x\ra^{-s}R_0^+(\lambda)\varphi\|_{L^2}\leq \int_{\mathbb{R}^3}\Big\|\frac{1}{\la x\ra^s4\pi|x-y|}\Big\|_{L_x^2}|\varphi(y)|dy\lesssim\|\varphi\|_{L^1}.
\end{aligned}
\end{equation}
Hence, $\lambda$ is an eigenvalue or a resonance (contradiction!).
\end{proof}

\begin{lemma}[Uniform bound for $(I+VR_0^+(\lambda))^{-1}$] If $V\in\mathcal{K}_0$ and $H$ has no eigenvalue or resonance on $[0,+\infty)$, then $S_\lambda:=(I+VR_0^+(\lambda))^{-1}: [0,+\infty)\to\mathcal{L}(L^1)$ is uniformly bounded.
\end{lemma}

\begin{proof} 
Iterating the resolvent identity, we get the formal identity
\begin{equation}
(I+VR_0^+(\lambda))^{-1}``="(I-VR_0^+(\lambda)+(VR_0^+(\lambda))^2-(VR_0^+(\lambda))^3)\sum_{n=0}^\infty (VR_0^+(\lambda))^{4n}.
\end{equation}
Indeed, by Lemma \ref{VR estimate} $(iii)$, $\|(VR_0^+(\lambda))^4\|_{L^1\to L^1}<\frac{1}{2}$ for all sufficiently large $\lambda$. Hence, the formal identity $(3.16)$ makes sense, and $(I+VR_0^+(\lambda))^{-1}$ is uniformly bounded for all sufficiently large $\lambda$. Thus, it suffices to show that $(I+VR_0^+(\lambda))^{-1}$ is continuous. To see this, we fix $\lambda_0\geq0$ and write
\begin{equation}
\begin{aligned}
&(I+VR_0^+(\lambda))^{-1}-(I+VR_0^+(\lambda_0))^{-1}=(I+VR_0^+(\lambda_0)+B_{\lambda,\lambda_0})^{-1}-S_{\lambda_0}\\
&=[(I+VR_0^+(\lambda_0)(I+S_{\lambda_0}B_{\lambda,\lambda_0})]^{-1}-S_{\lambda_0}=(I+S_{\lambda_0}B_{\lambda,\lambda_0})^{-1}S_{\lambda_0}-S_{\lambda_0}\\
``&="\sum_{n=0}^\infty (-S_{\lambda_0}B_{\lambda,\lambda_0})^nS_{\lambda_0}-S_{\lambda_0}``="\sum_{n=1}^\infty (-S_{\lambda_0}B_{\lambda,\lambda_0})^nS_{\lambda_0}.
\end{aligned}
\end{equation}
Then, by Lemma \ref{VR estimate} $(ii)$, we have
\begin{equation}
\begin{aligned}
&\|(I+VR_0^+(\lambda))^{-1}-(I+VR_0^+(\lambda_0))^{-1}\|_{L^1\to L^1}\leq\sum_{n=1}^\infty \|S_{\lambda_0}\|_{L^1\to L^1}^{n+1}\|B_{\lambda,\lambda_0}\|_{L^1\to L^1}^n\\
&=\frac{\|S_{\lambda_0}\|_{L^1\to L^1}^2\|B_{\lambda,\lambda_0}\|_{L^1\to L^1}}{1-\|S_{\lambda_0}\|_{L^1\to L^1}\|B_{\lambda,\lambda_0}\|_{L^1\to L^1}}\to 0\textup{ as }\lambda\to\lambda_0.
\end{aligned}
\end{equation}
Therefore, the formal identity $(3.17)$ makes sense, and $(I+VR_0^+(\lambda))^{-1}$ is continuous.
\end{proof}

\begin{lemma} If $V\in\mathcal{K}_0$ and $H$ has no eigenvalue or resonance on $[0,+\infty)$, then
$\tilde{S}_{\lambda}:=(S_\lambda -I)=(I+VR_0^+(\lambda))^{-1}-I:[0,+\infty)\to\mathcal{L}(L^1)$ is not only uniformly bounded but also an integral operator with kernel $\tilde{S}_\lambda(x,y)$:
\begin{equation}\label{tildeS}
\tilde{S}:=\sup_{\lambda\geq0}\|\tilde{S}_{\lambda}\|_{L^1\to L^1}=\sup_{\lambda\geq0}\|\tilde{S}_{\lambda}(x,y)\|_{L_y^\infty L_x^1}<\infty.
\end{equation}
\end{lemma}

\begin{proof}
By algebra, we have
\begin{equation}
\tilde{S}_{\lambda}=(I+VR_0^+(\lambda))^{-1}-I=-(I+VR_0^+(\lambda))^{-1}VR_0^+(\lambda)=-S_\lambda VR_0^+(\lambda).
\end{equation}
Since $\tilde{S}_{\lambda}: L^1\to L^1$ is bounded, sending $f_\epsilon\to\delta(\cdot-y_0)$ as $\epsilon\to0$, we get
\begin{equation}
(\tilde{S}_{\lambda}f_\epsilon)(x)=(-S_\lambda VR_0^+(\lambda)f_\epsilon)(x)\to -S_\lambda \Big(\frac{V(\cdot)e^{i\sqrt{\lambda}|\cdot-y_0|}}{4\pi|\cdot-y_0|}\Big)(x)=:\tilde{S}_\lambda(x,y_0).
\end{equation}
Consider $F_V(x;y,\lambda):=V(x)\frac{e^{i\sqrt{\lambda}|x-y|}}{4\pi|x-y|}$ as a function of $x$ with parameters $y\in\mathbb{R}^3$ and $\lambda\in\mathbb{R}$. Then, $F_V(x;y,\lambda)$ is bounded in $L_x^1$ uniformly in $y$ and $\lambda$. Therefore, by Lemma 3.3, we conclude that $\tilde{S}_\lambda(x,y)=-S_\lambda \Big(\frac{V(\cdot)e^{i\sqrt{\lambda}|\cdot-y|}}{4\pi|\cdot-y|}\Big)(x)$ is also bounded in $L_x^1$ uniformly in $\lambda$ and $y$.
\end{proof}

\subsection{Spectral projections and eigenfunctions}
Let $\chi$ be the dyadic partition of unity function chosen in \eqref{Hs}, and let $\tilde{\chi}_N(\lambda)\in C_c^\infty(\mathbb{R})$ such that $\tilde{\chi}_N(\lambda)=\chi(\tfrac{\sqrt{\lambda}}{N})$ if $\lambda\geq0$; $\tilde{\chi}_N(\lambda)=0$ if $\lambda<0$. By functional calculus, we define the Littlewood-Paley projections by $P_N=\tilde{\chi}_N(H)$, $P_{\leq N_0}=\sum_{N<N_0}P_N$, $P_{N_0<\cdot<N_1}=\sum_{N_0< N<N_1}P_N$ and $P_{\geq N_1}=\sum_{N\geq N_1}P_N$.

\begin{lemma}
Suppose that $V\in\mathcal{K}_0\cap L^{3/2,\infty}$ and $H$ has no eigenvalue or resonance on $[0,+\infty)$. Let $\mathfrak{S}:=\{f\in L^1\cap L^\infty:  P_cf= P_{N_0<\cdot<N_1} f \textup{ for some }N_0, N_1>0\}$. For $1<r<\infty$, $\mathfrak{S}$ is dense in $L^r$.
\end{lemma}

\begin{proof}
$L^1\cap L^\infty$ is dense in $L^r$. Fix $f\in L^1\cap L^\infty$. We claim that $\lim_{N_0\to0}\|P_{<N_0}f\|_{L^r}=0$. By the spectral theory, $\lim_{N_0\to0}\|P_{<N_0}f\|_{L^2}=0$. On the other hand, replacing $\tilde{\chi}_N$ by $\sum_{N<N_0}\tilde{\chi}_N$ in the proof of \cite[Corollary 1.6]{H1}, one can show that $\|P_{<N_0}f\|_{L^1}$ and $\|P_{<N_0}f\|_{L^\infty}$ are bounded uniformly in $N_0$. Hence the claim follows from interpolation. By the same argument, one can show that $\lim_{N_1\to\infty}\|P_{>N_1}f\|_{L^r}=0$. Thus, $\mathfrak{S}$ is dense in $L^r$.
\end{proof}

\begin{lemma}[Boundedness of eigenfunctions]\label{eigenfunctions} Suppose that $V\in\mathcal{K}_0\cap L^{3/2,\infty}$ and $H$ has no eigenvalue or resonance on $[0,+\infty)$. Let $\psi_j$ be an eigenfunction corresponding to the negative eigenvalue $\lambda_j$.\\
$(i)$ For all $1\leq p<\infty$, $\psi_j\in L^p$ and $P_{\lambda_j}$ is bounded on $L^p$, where $P_{\lambda_j}$ is the spectral projection onto the point $\{\lambda_j\}$.\\
$(ii)$ $\nabla\psi_j\in L^r$ for $1<r<3$. 
\end{lemma}

\begin{proof}
$(i)$ We prove the lemma following the argument in \cite{B}. We decompose $V=V_1+V_2$ such that $V_1$ is compactly supported and bounded, and $\|V_2\|_{\mathcal{K}}\leq1$. Then,
\begin{equation}
\begin{aligned}
&\psi_j+R_0(\lambda_j)V\psi_j=\psi_j+R_0(\lambda_j)(V_1+V_2)\psi_j=0\\
&\Rightarrow \psi_j=-(I+R_0(\lambda_j)V_2)^{-1}R_0(\lambda_j)V_1\psi_j=-\sum_{n=0}^\infty (-R_0(\lambda_j)V_2)^n R_0(\lambda_j)V_1\psi_j.
\end{aligned}
\end{equation}
Observe that, since $V_1$ is compactly supported, and $\lambda_j<0$, $R_0(\lambda_j)V_1\psi_j$ is exponentially decaying. To see this, we choose sufficiently small $\epsilon>0$ such that $\epsilon<\sqrt{-\lambda_j}$ for any negative eigenvalue $\lambda_j$. Indeed, there exists such $\epsilon$, since by the assumptions, there are at most finitely many negative eigenvalues (see \cite{BG}). Then, by the fractional integration inequality and the H\"older inequality in the Lorentz spaces (Lemma A.2), we get
\begin{equation}
\begin{aligned}
|e^{\epsilon|x|}(R_0(\lambda_j)V_1f)(x)|&\leq e^{\epsilon|x|}\int_{\mathbb{R}^3}\frac{e^{i\sqrt{\lambda_j}|x-y|}}{4\pi|x-y|}|V_1(y)||\psi_j(y)|dy\\
&\leq \int_{\mathbb{R}^3}\frac{e^{-(\sqrt{-\lambda_j}-\epsilon)|x-y|}}{4\pi|x-y|}e^{\epsilon|y|}|V_1(y)||\psi_j(y)|dy\\
&\leq \|e^{\epsilon|\cdot|}V_1\psi_j\|_{L^{3/2,1}}\lesssim\|e^{\epsilon|\cdot|}V_1\|_{L^{6,2}}\|\psi_j\|_{L^2}.
\end{aligned}
\end{equation}
Similarly, one can check that $e^{\epsilon|\cdot|}R_0(\lambda_j)V_2e^{-\epsilon|\cdot|}$ is bounded on $L^\infty$ and its operator norm is strictly less than $1$. Thus, we prove that
\begin{equation}
\|e^{\epsilon|\cdot|}\psi_j\|_{L^\infty}\leq\Big(\sum_{n=0}^\infty \|e^{\epsilon|\cdot|}R_0(\lambda_j)V_2e^{-\epsilon|\cdot|}\|_{L^\infty\to L^\infty}^n\Big)\|e^{\epsilon|\cdot|}R_0(\lambda_j)V_1\psi_j\|_{L^\infty}<\infty.
\end{equation}
Therefore, $\psi_j\in L^p$ and $P_{\lambda_j}f=\la\psi_j, f\ra_{L^2}\psi_j$ is bounded on $L^p$ for all $1\leq p\leq\infty$.\\
$(ii)$ Let $\delta_1,\delta_2>0$ be arbitrarily small numbers. Then, since $\lambda_j<0$, by the inhomogeneous Sobolev inequality, we get
\begin{equation}
\begin{aligned}
\|\nabla\psi_j\|_{L^{\frac{1}{1-\delta_1}}}&=\|\nabla R_0^+(\lambda_j)V\psi_j\|_{L^{\frac{1}{1-\delta_1}}}\lesssim\|V\psi_j\|_{L^{\frac{1}{1-\delta_1}}}\\
&\leq \|V\|_{L^{3/2,\infty}}\|\psi_j\|_{L^{\frac{3}{1-3\delta_1},1}}<\infty,\\
\|\nabla\psi_j\|_{L^{\frac{3}{1+\delta_2}}}&=\|\nabla R_0^+(\lambda_j)V\psi_j\|_{L^{\frac{3}{1+\delta_2}}}\lesssim\|V\psi_j\|_{W^{-1,\frac{3}{1+\delta_2}}}\\
&\lesssim\|V\psi_j\|_{L^{\frac{3}{2+\delta_2}}}\leq \|V\|_{L^{3/2,\infty}}\|\psi_j\|_{L^{\frac{3}{\delta_2},\frac{3}{2+\delta_2}}}<\infty.
\end{aligned}
\end{equation}
Thus, interpolation gives $(ii)$.
\end{proof}

\section{High Frequency Estimate: Proof of Lemma \ref{Key} $(i)$}

\subsection{Construction of the formal series expansion}
For a large dyadic number $N_1$ to be chosen later, we construct a formal series for $\textup{Pb}_{\geq N_1}$ as follows. First, iterating the resolvent identity
\begin{equation}
(I+VR_0^+(\lambda))^{-1}=I-(I+VR_0^+(\lambda))^{-1}VR_0^+(\lambda),
\end{equation}
we generate a formal series expansion
\begin{equation}\label{formal series1: high}
(I+VR_0^+(\lambda))^{-1}``="\sum_{n=0}^\infty(-VR_0^+(\lambda))^n.
\end{equation}
Plugging \eqref{formal series1: high} into \eqref{PbN}, we write 
\begin{equation}\label{formal series2: high}
\textup{Pb}_{\geq N_1}``="-\sum_{N\geq N_1}\sum_{n=0}^\infty \frac{1}{\pi}\int_0^\infty m(\lambda)\chi_N(\sqrt{\lambda})\Im[R_0^+(\lambda)(-VR_0^+(\lambda))^nVR_0^+(\lambda)]d\lambda.
\end{equation}
Then, writing the first and the last free resolvents explicitly by the free resolvent formula $R_0^+(\lambda)(x,y)=\frac{e^{i\sqrt{\lambda}|x-y|}}{4\pi|x-y|}$ and collecting terms having $\lambda$ by Fubini theorem, we write the kernel of $\textup{Pb}_{\geq N_1}$ as
\begin{equation}\label{formal series3: high}
\begin{aligned}
&\textup{Pb}_{\geq N_1}(x,y)\\
``&="\sum_{N\geq N_1}\sum_{n=0}^\infty \frac{(-1)^{n+1}}{\pi}\int_0^\infty m(\lambda)\chi_N(\sqrt{\lambda})\\
&\quad\quad\quad\quad \times\Im\Big[\int_{\mathbb{R}^6}\frac{e^{i\sqrt{\lambda}|x-\tilde{x}|}}{4\pi|x-\tilde{x}|}(VR_0^+(\lambda))^n(\tilde{x},\tilde{y})V(\tilde{y})\frac{e^{i\sqrt{\lambda}|\tilde{y}-y|}}{4\pi|\tilde{x}-y|}d\tilde{x}d\tilde{y}\Big]d\lambda\\
&=\int_{\mathbb{R}^6}\frac{V(\tilde{y})}{16\pi^3|x-\tilde{x}||\tilde{y}-y|}\Big\{\sum_{N\geq N_1}\sum_{n=0}^\infty (-1)^{n+1}\textup{Pb}_N^n(x,\tilde{x},\tilde{y}, y)\Big\}d\tilde{x}d\tilde{y},
\end{aligned}
\end{equation}
where
\begin{equation}\label{PbNn}
\textup{Pb}_N^n(x,\tilde{x},\tilde{y}, y)=\int_0^\infty m(\lambda)\chi_N(\sqrt{\lambda})\Im[e^{i\sqrt{\lambda}(|x-\tilde{x}|+|\tilde{y}-y|)}(VR_0^+(\lambda))^n(\tilde{x},\tilde{y})]d\lambda.
\end{equation}
We note that the series \eqref{formal series3: high} makes sense only formally at this moment, but it will be shown that the sum is absolutely convergent, and that it satisfies the bound we want to have.

\subsection{Intermediate kernel estimates}
We estimate the intermediate kernel $\textup{Pb}_N^n(x,\tilde{x},\tilde{y}, y)$ in two ways. First, we show that the sum of $\textup{Pb}_N^n(x,\tilde{x},\tilde{y}, y)$ in $N\geq N_1$ is absolutely convergent, and moreover each $\textup{Pb}_N^n(x,\tilde{x},\tilde{y}, y)$ decays away from $x=\tilde{x}$ and $\tilde{y}=y$.

\begin{lemma}[Summability in $N$]\label{summability in N: high} 
For $s_1,s_2\geq0$, we have
\begin{equation}\label{eq1: summability in N: high}
|\textup{Pb}_N^n (x, \tilde{x}, \tilde{y}, y)|\lesssim \frac{N^2\|m\|_{\mathcal{H}(s_1+s_2)}}{\la N(x-\tilde{x})\ra^{s_1}\la N(\tilde{y}-y)\ra^{s_2}}k_1^n(\tilde{x},\tilde{y})
\end{equation}
and
\begin{equation}\label{eq2: summability in N: high}
\|k_1^n(\tilde{x},\tilde{y})\|_{L_{\tilde{y}}^\infty L_{\tilde{x}}^1}\leq\Big(\frac{\|V\|_{\mathcal{K}}}{4\pi}\Big)^n.
\end{equation}
\end{lemma}

For the proof, we need the following lemma.
\begin{lemma}[Oscillatory integral]\label{oscillatory} For $s\geq0$,
\begin{equation}
\Big|\int_0^\infty m(\lambda)\chi_N(\sqrt{\lambda})\Im(e^{i\sqrt{\lambda}\sigma}) d\lambda\Big|\lesssim \frac{N^2}{\la N\sigma \ra^s}\|m\|_{\mathcal{H}(s)}.
\end{equation}
\end{lemma}

\begin{proof}
By abuse of notation, we denote by $\chi$ the even extension of itself. Making change of variables $\lambda\mapsto N^2\lambda^2$, we write
\begin{equation}
\begin{aligned}
\int_0^\infty m(\lambda)\chi_N(\sqrt{\lambda})\Im(e^{i\sqrt{\lambda}\sigma}) d\lambda&=N^2\int_0^\infty 2\lambda m(N^2\lambda^2)\chi(\lambda)\sin(N\lambda\sigma)  d\lambda\\
&=N^2\int_\mathbb{R}\lambda m(N^2\lambda^2)\chi(\lambda)e^{i\lambda N\sigma}d\lambda\\
&=N^2 \Big(m(N^2\lambda^2)\lambda\chi(\lambda)\Big)^\vee (N\sigma)\\
&=\frac{N^2}{\la N\sigma\ra^s}\Big(\la\nabla\ra^s(m(N^2\lambda^2)\lambda\chi(\lambda))\Big)^\vee (N\sigma).
\end{aligned}
\end{equation}
Thus, it follows from Hausdorff-Young inequality and the fractional Leibniz rule that
\begin{equation}
\begin{aligned}
\Big|\int_0^\infty m(\lambda)\chi_N(\sqrt{\lambda})\Im(e^{i\sqrt{\lambda}\sigma}) d\lambda\Big|&\leq\frac{N^2}{\la N\sigma\ra^s}\|m(N^2\lambda^2)\lambda\chi(\lambda)\|_{W^{s,1}}\\
&\lesssim\frac{N^2}{\la N\sigma\ra^s}\|m(N^2\lambda^2)\chi(\lambda)\|_{W^{s,2}}\\&\leq\frac{N^2}{\la N\sigma\ra^s}\|m\|_{\mathcal{H}(s)}.
\end{aligned}
\end{equation}
\end{proof}

\begin{proof}[Proof of Lemma \ref{summability in N: high}]
First, using the free resolvent formula, we write
\begin{equation}
\begin{aligned}
&\textup{Pb}_N^n (x,\tilde{x},\tilde{y},y)\\
&=\int_0^\infty m(\lambda)\chi_N(\sqrt{\lambda})\Im\Big\{\int_{\mathbb{R}^{3(n-1)}}\prod_{k=1}^n V(x_k)\frac{\prod_{k=0}^{n+1} e^{i\sqrt{\lambda}|x_k-x_{k+1}|}}{\prod_{k=1}^n 4\pi|x_k-x_{k+1}|}d\mathbf{x}_{(2,n)}\Big\} d\lambda\\
&=\int_{\mathbb{R}^{3(n-1)}}\frac{\prod_{k=1}^n V(x_k)}{\prod_{k=1}^n 4\pi|x_k-x_{k+1}|} \Big\{\int_0^\infty m(\lambda)\chi_N(\sqrt{\lambda})\Im(e^{i\sqrt{\lambda}\sigma_{n+1}}) d\lambda\Big\} d\mathbf{x}_{(2,n)},
\end{aligned}
\end{equation}
where $x_0:=x$, $x_1:=\tilde{x}$, $x_{n+1}:=\tilde{y}$, $x_{n+2}:=y$, $d\mathbf{x}_{(2,n)}:=dx_2\cdot\cdot\cdot dx_n$ and $\sigma_n:=\sum_{j=0}^n|x_j-x_{j+1}|$. Then, by Lemma \ref{oscillatory} with $s=s_1+s_2$ and the trivial inequality
\begin{equation}
|x_0-x_1|=|x-\tilde{x}|, |x_{n+1}-x_{n+2}|=|\tilde{y}-y|\leq \sigma_{n+1}=\sum_{j=0}^{n+1}|x_j-x_{j+1}|,
\end{equation}
we obtain that 
\begin{equation}
|\textup{Pb}_N^n (x, \tilde{x}, \tilde{y}, y)|\lesssim \frac{N^2\|m\|_{\mathcal{H}(s_1+s_2)}}{\la N(x-\tilde{x})\ra^{s_1}\la N(\tilde{y}-y)\ra^{s_2}}\int_{\mathbb{R}^{3(n-1)}}\frac{\prod_{k=1}^n |V(x_k)|}{\prod_{k=1}^n 4\pi|x_k-x_{k+1}|} d\mathbf{x}_{(2,n)}.
\end{equation}
We define
\begin{equation}
k_1^n(\tilde{x},\tilde{y}):=\int_{\mathbb{R}^{3(n-1)}}\frac{\prod_{k=1}^n |V(x_k)|}{\prod_{k=1}^n 4\pi|x_k-x_{k+1}|} d\mathbf{x}_{(2,n)}.
\end{equation}
Then, by the definition of the global Kato norm, we have
\begin{equation}
\begin{aligned}
\|k_1^n(\tilde{x},\tilde{y})\|_{L_{\tilde{y}}^\infty L_{\tilde{x}}^1}&\leq\sup_{x_{n+1}\in\mathbb{R}^3}\int_{\mathbb{R}^{3n}}\frac{\prod_{k=1}^n |V(x_k)|}{\prod_{k=1}^n 4\pi|x_k-x_{k+1}|} d\mathbf{x}_{(1,n)}\\
&\leq\Big(\sup_{x_{n}\in\mathbb{R}^3}\int_{\mathbb{R}^{3(n-1)}}\frac{\prod_{k=1}^{n-1} |V(x_k)|}{\prod_{k=1}^{n-1} 4\pi|x_k-x_{k+1}|} d\mathbf{x}_{(1,n-1)}\Big)\\
&\quad\quad\times\Big(\sup_{x_{n+1}\in\mathbb{R}^3}\int_{\mathbb{R}^3}\frac{|V(x_n)|}{4\pi|x_n-x_{n+1}|}dx_n\Big)\\
&\leq\Big(\sup_{x_{n}\in\mathbb{R}^3}\int_{\mathbb{R}^{3(n-1)}}\frac{\prod_{k=1}^{n-1} |V(x_k)|}{\prod_{k=1}^{n-1} 4\pi|x_k-x_{k+1}|} d\mathbf{x}_{(1,n-1)}\Big)\Big(\frac{\|V\|_{\mathcal{K}}}{4\pi}\Big)\\
&\leq\cdots(\textup{repeat})\cdots\leq\Big(\frac{\|V\|_{\mathcal{K}}}{4\pi}\Big)^n.
\end{aligned}
\end{equation}
\end{proof}

Next, we show summability of the intermediate kernel in $n$.

\begin{lemma}[Summability in $n$]\label{summability in n: high}
For $\epsilon>0$, there exist $N_1=N_1(V,\epsilon)\gg1$ and $k_2^n(\tilde{x},\tilde{y})\in L_{\tilde{y}}^\infty L_{\tilde{x}}^1$ such that for $N\geq N_1$,
\begin{equation}\label{eq1: summability in n: high}
|\textup{Pb}_N^n (x, \tilde{x}, \tilde{y}, y)|\lesssim \epsilon^nN^2 \|m\|_{\mathcal{H}(0)}k_2^n(\tilde{x},\tilde{y}).
\end{equation}
and
\begin{equation}\label{eq2: summability in n: high}
\|k_2^n(\tilde{x},\tilde{y})\|_{L_{\tilde{y}}^\infty L_{\tilde{x}}^1}\lesssim \epsilon^n.
\end{equation}
\end{lemma}

\begin{proof}
By Lemma \ref{VR estimate} $(iii)$, given $\epsilon>0$, there exist $N_1\gg1$ and an operator $D:L^1\to L^1$ such that $\|D(x,y)\|_{L_y^\infty L_x^1}\leq \epsilon^4$ and $|(VR_0^+(\lambda))^4(x,y)|\leq D(x,y)$. We also observe that
\begin{equation}
|(VR_0^+(\lambda))(x,y)|=\Big|V(x)\frac{e^{i\sqrt{\lambda}|x-y|}}{4\pi|x-y|}\Big|=\frac{|V(x)|}{4\pi|x-y|}=\Big(|V|(-\Delta)^{-1}\Big)(x,y).
\end{equation}
We denote by $\lfloor a\rfloor$ the largest integer less than or equal to $a$. Then, we have
\begin{equation}
\begin{aligned}
|\textup{Pb}_N^n(x,\tilde{x},\tilde{y}, y)|&\leq \int_0^\infty |m(\lambda)|\chi_N(\sqrt{\lambda})|(VR_0^+(\lambda))^n(\tilde{x},\tilde{y})|d\lambda\\
&\leq\int_0^\infty |m(\lambda)|\chi_N(\sqrt{\lambda})\Big|\Big(D^{\lfloor\frac{n}{4}\rfloor}(|V|(-\Delta)^{-1})^{n-4\lfloor\frac{n}{4}\rfloor}\Big) (\tilde{x},\tilde{y})\Big|d\lambda\\
&\lesssim N^2\int_0^\infty |m(N^2\lambda)|\chi(\lambda)d\lambda \cdot \Big|\Big(D^{\lfloor\frac{n}{4}\rfloor}(|V|(-\Delta)^{-1})^{n-4\lfloor\frac{n}{4}\rfloor}\Big) (\tilde{x},\tilde{y})\Big|\\
&\lesssim N^2 \|m\|_{\mathcal{H}(0)}\Big|\Big(D^{\lfloor\frac{n}{4}\rfloor}(|V|(-\Delta)^{-1})^{n-4\lfloor\frac{n}{4}\rfloor}\Big) (\tilde{x},\tilde{y})\Big|.
\end{aligned}
\end{equation}
We define 
\begin{equation}
k_2^n(\tilde{x},\tilde{y}):=\Big|\Big(D^{\lfloor\frac{n}{4}\rfloor}(|V|(-\Delta)^{-1})^{n-4\lfloor\frac{n}{4}\rfloor}\Big)(\tilde{x},\tilde{y})\Big|.
\end{equation}
Then, by Lemma \ref{VR estimate}, one can check \eqref{eq2: summability in n: high}.
\end{proof}

\subsection{Proof of Lemma \ref{Key} $(i)$}
Let $\delta>0$ be a sufficiently small number to be chosen later. Let $\epsilon>0$ be a small number depending on $\|V\|_{\mathcal{K}}$ and $\delta>0$ (see $(4.30)$). Then, we pick a large dyadic number $N_1$ from Lemma \ref{summability in n: high}. We will show that $\textup{Pb}_{\geq N_1}$ is bounded from $L^{\frac{3}{3-\delta},1}$ to
$L^{\frac{3}{3-\delta},\infty}$.

Let $s=\frac{2}{1-\delta}>2$ and $\theta=\frac{2-\delta}{2}$ ($\Rightarrow2\theta=2-\delta$, $(s-2)\theta>\delta$ and $s\theta>2$). Then, by Lemma \ref{summability in N: high} with $s_1=2$ and $s_2=s-2$ and Lemma \ref{summability in n: high}, we get
\begin{equation}\label{combined}
\begin{aligned}
|\textup{Pb}_N^n (x, \tilde{x}, \tilde{y}, y)|&=|\textup{Pb}_N^n (x, \tilde{x}, \tilde{y}, y)|^{\theta}|\textup{Pb}_N^n (x, \tilde{x}, \tilde{y}, y)|^{1-\theta}\\
&\lesssim \frac{N^2 \|m\|_{\mathcal{H}(s)}}{\la N(x-\tilde{x})\ra^{2\theta}\la N(\tilde{y}-y)\ra^{(s-2)\theta}}\Big(k_1^n(\tilde{x},\tilde{y})\Big)^{\theta}\Big(k_2^n(\tilde{x},\tilde{y})\Big)^{1-\theta}.
\end{aligned}
\end{equation}
We claim that 
\begin{equation}\label{sum estimate}
\sum_{N\in 2^{\mathbb{Z}}}\frac{N^2}{\la Nx\ra^{2\theta}\la Ny\ra^{(s-2)\theta}}\lesssim \frac{1}{|x|^{2-\delta}|y|^{\delta}}.
\end{equation}
Fix $x, y\in\mathbb{R}^3$, and consider the following four cases.\\
\textbf{(Case 1: $N<\min(|x|^{-1}, |y|^{-1})$)}
\begin{equation}
\sum_{\textup{Case 1}}\frac{N^2}{\la Nx\ra^{2\theta}\la Ny\ra^{(s-2)\theta}}\leq\sum_{\textup{Case 1}}N^2\leq\min\Big(\frac{1}{|x|}, \frac{1}{|y|}\Big)^2\leq\frac{1}{|x|^{2-\delta}|y|^{\delta}}.
\end{equation}
\textbf{(Case 2: $|x|^{-1}\leq N<|y|^{-1}$)}
\begin{equation}
\begin{aligned}
\sum_{\textup{Case 2}}\frac{N^2}{\la Nx\ra^{2\theta}\la Ny\ra^{(s-2)\theta}}&\leq \sum_{\textup{Case 2}}\frac{N^2}{|Nx|^{2\theta}}=\sum_{\textup{Case 2}}\frac{N^{2(1-\theta)}}{|x|^{2\theta}}=\sum_{\textup{Case 2}}\frac{N^\delta}{|x|^{2-\delta}}\\
&\leq\frac{1}{|x|^{2-\delta}|y|^{\delta}}.
\end{aligned}
\end{equation}
\textbf{(Case 3: $|y|^{-1}\leq N<|x|^{-1}$)}
\begin{equation}
\sum_{\textup{Case 3}}\frac{N^2}{\la Nx\ra^{2\theta}\la Ny\ra^{(s-2)\theta}}\leq \sum_{\textup{Case 3}}\frac{N^2}{|Ny|^{\delta}}=\sum_{\textup{Case 3}}\frac{N^{2-\delta}}{|y|^{\delta}}\leq \frac{1}{|x|^{2-\delta}|y|^{\delta}}.
\end{equation}
\textbf{(Case 4: $N\geq\max(|x|^{-1}, |y|^{-1})$)}
\begin{equation}
\begin{aligned}
\sum_{\textup{Case 4}}\frac{N^2}{\la Nx\ra^{2\theta}\la Ny\ra^{(s-2)\theta}}&\lesssim \frac{1}{|x|^{2\theta}|y|^{(s-2)\theta}}\sum_{\textup{Case 4}}\frac{1}{N^{s\theta-2}}\\
&\leq\frac{1}{|x|^{2-\delta}|y|^{(s-2)\theta}}|y|^{s\theta-2}\leq\frac{1}{|x|^{2-\delta}|y|^{2-2\theta}}=\frac{1}{|x|^{2-\delta}|y|^{\delta}}.
\end{aligned}
\end{equation}
Collecting all, we prove the claim.

Applying \eqref{sum estimate} to \eqref{combined} and summing in $N\geq N_1$, we obtain
\begin{equation}
\sum_{N\geq N_1}|\textup{Pb}_N^n (x, \tilde{x}, \tilde{y}, y)|\lesssim \frac{\|m\|_{\mathcal{H}(s)}}{ |x-\tilde{x}|^{2-\delta}|\tilde{y}-y|^\delta}\Big(k_1^n(\tilde{x},\tilde{y})\Big)^{\theta}\Big(k_2^n(\tilde{x},\tilde{y})\Big)^{1-\theta}.
\end{equation} 
Let
\begin{equation}
K(x,y)=\sum_{n=0}^\infty\Big(k_1^n(\tilde{x},\tilde{y})\Big)^{\theta}\Big(k_2^n(x,y)\Big)^{1-\theta}.
\end{equation}
Then, $K\in L_y^\infty L_x^1$, since if $N_1$ is large enough,
\begin{equation}\label{K bound}
\begin{aligned}
\|K(x,y)\|_{L_y^\infty L_x^1}&\leq\sum_{n=0}^\infty\Big\|\Big(k_1^n(\tilde{x},\tilde{y})\Big)^{\theta}\Big(k_2^n(x,y)\Big)^{1-\theta}\Big\|_{L_y^\infty L_x^1}\\
&\leq\sum_{n=0}^\infty\|k_1^n(\tilde{x},\tilde{y})\|_{L_y^\infty L_x^1}^\theta\|k_2^n(x,y)\|_{L_y^\infty L_x^1}^{1-\theta}\\
&\leq\sum_{n=0}^\infty\Big(\frac{\|V\|_{\mathcal{K}}}{4\pi}\Big)^{n\theta}\epsilon^n<\infty,
\end{aligned}
\end{equation}
where $\epsilon>0$ is chosen so that $(\frac{\|V\|_{\mathcal{K}}}{4\pi})^{\theta}\epsilon<1$ with $\theta=\frac{2-\delta}{2}$. Therefore, we obtain the kernel estimates for $\textup{Pb}_{\geq N_1}(x,y)$,
\begin{equation}
\begin{aligned}
|\textup{Pb}_{N_1}(x,y)|&\leq \int_{\mathbb{R}^6}\frac{V(\tilde{y})}{16\pi^3|x-\tilde{x}||\tilde{y}-y|}\sum_{N\geq N_1}\sum_{n=0}^\infty |\textup{Pb}_N^n (x, \tilde{x}, \tilde{y}, y)|d\tilde{x} d\tilde{y}\\
&\lesssim\int_{\mathbb{R}^6}\frac{|V(\tilde{y})|}{|x-\tilde{x}|^{3-\delta}|\tilde{y}-y|^{1+\delta}}K(\tilde{x},\tilde{y})d\tilde{x}d\tilde{y},
\end{aligned}
\end{equation}
with $K\in L_y^\infty L_x^1$.

Let $T_K$ be the integral operator with kernel $K(x,y)$, which is bounded on $L^1$ by \eqref{K bound}. By the fractional integration inequality and H\"older inequality in the Lorentz spaces (see Appendix A), we conclude that
\begin{equation}\label{key estimates}
\begin{aligned}
&\|\textup{Pb}_{\geq N_1}f\|_{L^{\frac{3}{3-\delta},\infty}}\\
&\lesssim\Big\|\int_{\mathbb{R}^9}\frac{|V(\tilde{y})|}{|x-\tilde{x}|^{3-\delta}|\tilde{y}-y|^{1+\delta}}K(\tilde{x},\tilde{y})|f(y)|d\tilde{x}d\tilde{y}dy\Big\|_{L_x^{\frac{3}{3-\delta},\infty}}\\
&\lesssim\||\nabla|^{-\delta}T_K(|V||\nabla|^{-(2-\delta)}(|f|))\|_{L^{\frac{3}{3-\delta},\infty}}\lesssim\|T_K(|V||\nabla|^{-(2-\delta)}(|f|))\|_{L_{x}^1}\\
&\lesssim\||V||\nabla|^{-(2-\delta)}(|f|)\|_{L_{x}^1}\leq\|V\|_{L^{3/2,\infty}}\||\nabla|^{-(2-\delta)}|f|\|_{L^{3,1}}\lesssim\|f\|_{L^{\frac{3}{3-\delta},1}}.
\end{aligned}
\end{equation}

\begin{remark}
In \eqref{key estimates}, we only used the fractional integration inequality and the H\"older inequality.  Note that after applying the fractional integration inequality, we always have the $L^{p,q}$-norm with smaller $p$ on the right hand side, although we want to show the $L^{\frac{3}{3-\epsilon},1}-L^{\frac{3}{3-\epsilon},\infty}$ boundedness. Hence, one must have at least one chance to raise the number $p$ to compensate the decrease of $p$ caused by the fractional integration inequalities. In \eqref{key estimates}, the potential $V$ plays such a role with the H\"older inequality. This is the main reason that we keep one extra potential term $V$ in the spectral representation by considering the perturbation $m(H)P_c-m(-\Delta)$ instead of $m(H)P_c$, and introducing  intermediated kernels $\textup{Pb}_N^n(x,\tilde{x},\tilde{y},y)$, even though they look rather artificial.
\end{remark}

\section{Low Frequency Estimate: Proof of Lemma \ref{Key} $(ii)$}

\subsection{Construction of the formal series expansion}
We prove Lemma \ref{Key} $(ii)$ by modifying the argument in Section 4. Note that for small $N$, the formal series expansion \eqref{formal series1: high} may not be convergent, since $(VR_0^+(\lambda))^4$ in \eqref{formal series1: high} is not small anymore. Hence, we introduce a new series expansion for $(I+VR_0^+(\lambda))^{-1}$,
\begin{equation}\label{formal series1: low}
\begin{aligned}
(I+VR_0^+(\lambda))^{-1}&=(I+VR_0^+(\lambda_0)+B_{\lambda,\lambda_0})^{-1}\\
&=[(I+B_{\lambda,\lambda_0}S_{\lambda_0})(I+VR_0^+(\lambda_0))]^{-1}\\
&=(I+VR_0^+(\lambda_0))^{-1}(I+B_{\lambda,0}S_{\lambda_0})^{-1}\\
``&="S_{\lambda_0}\sum_{n=0}^\infty (-B_{\lambda,\lambda_0}S_{\lambda_0})^n,
\end{aligned}
\end{equation}
where $B_{\lambda,\lambda_0}=V(R_0^+(\lambda)-R_0^+(\lambda_0))$ and $S_{\lambda_0}=(I+VR_0^+(\lambda_0))^{-1}$. Plugging the formal series $(\ref{formal series1: low})$ with $\lambda_0=0$ into \eqref{PbN}, we write 
\begin{equation}\label{formal series2: low}
\textup{Pb}_N``="\sum_{n=0}^\infty \frac{(-1)^{n+1}}{\pi}\int_0^\infty m(\lambda)\chi_N(\sqrt{\lambda})\Im[R_0^+(\lambda)S_{0}(B_{\lambda,0}S_{0})^nVR_0^+(\lambda)]d\lambda.
\end{equation}
As in the previous section, writing the first and the last free resolvents explicitly by the free resolvent formula $R_0^+(\lambda)(x,y)=\frac{e^{i\sqrt{\lambda}|x-y|}}{4\pi|x-y|}$ and collecting terms having $\lambda$ by Fubini theorem, we write the kernel of $\textup{Pb}_N$ as
\begin{equation}\label{formal series3: low}
\begin{aligned}
&\textup{Pb}_N(x,y)\\
``&="\sum_{n=0}^\infty \frac{(-1)^{n+1}}{\pi}\int_0^\infty m(\lambda)\chi_N(\sqrt{\lambda})\\
&\quad\quad\quad\quad\times\Im\Big[\iint_{\mathbb{R}^6}\frac{e^{i\sqrt{\lambda}|x-\tilde{x}|}}{4\pi|x-\tilde{x}|}[S_{0}(B_{\lambda,0}S_{0})^n](\tilde{x},\tilde{y})V(\tilde{y})\frac{e^{i\sqrt{\lambda}|\tilde{y}-y|}}{4\pi|\tilde{x}-y|}d\tilde{x}d\tilde{y}\Big]d\lambda\\
&=\iint_{\mathbb{R}^6}\frac{V(\tilde{y})}{16\pi^3|x-\tilde{x}||\tilde{y}-y|}\Big[\sum_{n=0}^\infty (-1)^{n+1}\textup{Pb}_N^n(x,\tilde{x},\tilde{y}, y)\Big]d\tilde{x}d\tilde{y},
\end{aligned}
\end{equation}
where
\begin{equation}\label{PbNn: low}
\textup{Pb}_N^n(x,\tilde{x},\tilde{y}, y)=\int_0^\infty m(\lambda)\chi_N(\sqrt{\lambda})\Im[e^{i\sqrt{\lambda}(|x-\tilde{x}|+|\tilde{y}-y|)}[S_{0}(B_{\lambda,0}S_{0})^n](\tilde{x},\tilde{y})]d\lambda.
\end{equation}
By Lemma \ref{VR estimate} $(ii)$, $B_{\lambda,0}$ in \eqref{formal series3: low} is small for sufficiently small $N$. This fact will guarantee the convergence of the formal series.

\subsection{Intermediate kernel estimates}
We will show the kernel estimates analogous to Lemma \ref{summability in N: high} and \ref{summability in n: high}. Then, Lemma \ref{Key} $(ii)$ will follow from exactly the same argument in Section 4.3, thus we omit the proof.

\begin{lemma}[Summability in $N$]\label{summability in N: low} There exists $k_1^n(\tilde{x},\tilde{y})$ such that for $s_1, s_2\geq0$,
\begin{equation}\label{eq1: summability in N: low}
|\textup{Pb}_N^n (x, \tilde{x}, \tilde{y}, y)|\lesssim \frac{N^2\|m\|_{\mathcal{H}(s_1+s_2)}}{\la N(x-\tilde{x})\ra^{s_1}\la N(\tilde{y}-y)\ra^{s_2}}k_1^n(\tilde{x},\tilde{y})
\end{equation}
and
\begin{equation}\label{eq2: summability in N: low}
\|k_1^n(\tilde{x},\tilde{y})\|_{L_{\tilde{y}}^\infty L_{\tilde{x}}^1}\leq(\tilde{S}+1)^{n+1}\Big(\frac{\|V\|_{\mathcal{K}}}{2\pi}\Big)^n,
\end{equation}
where $\tilde{S}$ is the positive number given by $(\ref{tildeS})$.
\end{lemma}

\begin{proof}
First, splitting $B_{\lambda,0}$ into $VR_0^+(\lambda)-VR_0^+(0)$ in
\begin{equation}
\textup{Pb}_N^n(x,\tilde{x},\tilde{y}, y)=\int_0^\infty m(\lambda)\chi_N(\sqrt{\lambda})\Im[e^{i\sqrt{\lambda}(|x-\tilde{x}|+|\tilde{y}-y|)}[S_{0}(B_{\lambda,0}S_{0})^n](\tilde{x},\tilde{y})]d\lambda,
\end{equation}
we write $\textup{Pb}_N^n(x,\tilde{x},\tilde{y}, y)$ as the sum of $2^n$ copies of
\begin{equation}\label{eq:piece}
\int_0^\infty m(\lambda)\chi_N(\sqrt{\lambda})\Im[e^{i\sqrt{\lambda}(|x-\tilde{x}|+|\tilde{y}-y|)}[S_0VR_0^+(\alpha_1\lambda)S_0\cdot\cdot\cdot VR_0^+(\alpha_n\lambda)S_0](\tilde{x},\tilde{y})]d\lambda
\end{equation}
up to $\pm$, where $\alpha_k=0\textup{ or }1$ for each $k=1, ..., n$. Next, splitting all $S_0$ into $I$ and $\tilde{S}_0$ in $(\ref{eq:piece})$, we further decompose $(\ref{eq:piece})$ into the sum of $2^{n+1}$ kernels.

Among them, let us consider the two representative terms,
\begin{align}
&\Im\int_0^\infty m(\lambda)\chi_N(\sqrt{\lambda})e^{i\sqrt{\lambda}(|x-\tilde{x}|+|\tilde{y}-y|)}[\tilde{S}_0VR_0^+(\alpha_1\lambda)\tilde{S}_0\cdot\cdot\cdot VR_0^+(\alpha_n\lambda)\tilde{S}_0](\tilde{x},\tilde{y})d\lambda,\label{example1}\\
&\Im\int_0^\infty m(\lambda)\chi_N(\sqrt{\lambda})e^{i\sqrt{\lambda}(|x-\tilde{x}|+|\tilde{y}-y|)}[VR_0^+(\alpha_1\lambda)\cdot\cdot\cdot VR_0^+(\alpha_n\lambda)](\tilde{x},\tilde{y})d\lambda.\label{example2}
\end{align}
For the first term, by the free resolvent formula \eqref{free resolvent formula}, we write \eqref{example1} in the integral form,
\begin{equation}
\begin{aligned}
&\Im\int_0^\infty \int_{\mathbb{R}^{6n}}m(\lambda)\chi_N(\sqrt{\lambda})\prod_{k=1}^{n+1}\tilde{S}_0(x_{2k-1}, x_{2k})\\
&\quad\quad\quad\quad\quad\quad\quad\quad\quad\quad\times\prod_{k=1}^n V(x_{2k})\frac{\prod_{k=0}^{n+1} e^{i{\alpha_k\sqrt{\lambda}}|x_{2k}-x_{2k+1}|}}{\prod_{k=1}^n 4\pi|x_{2k}-x_{2k+1}|}d\mathbf{x}_{(2,2n+1)} d\lambda\\
&=\int_{\mathbb{R}^{6n}}\frac{\prod_{k=1}^{n+1}\tilde{S}_0(x_{2k-1}, x_{2k})\prod_{k=1}^n V(x_{2k})}{\prod_{k=1}^n 4\pi|x_{2k}-x_{2k+1}|}\\
&\quad\quad\quad\quad\quad\quad\quad\quad\quad\quad\times\Big\{\int_0^\infty m(\lambda)\chi_N(\sqrt{\lambda})\Im(e^{i\sqrt{\lambda}\tilde{\sigma}_{n+1}})d\lambda\Big\}d\mathbf{x}_{(2,2n+1)}
\end{aligned}
\end{equation}
where $x_0:=x$, $x_1:=\tilde{x}$, $x_{2n+2}:=\tilde{y}$, $x_{2n+3}:=y$, $d\mathbf{x}_{(2,n)}:=dx_2\cdot\cdot\cdot dx_n$, $\tilde{\sigma}_n:=\sum_{k=0}^{n+1} \alpha_k|x_{2k}-x_{2k+1}|$ and $\alpha_0=\alpha_{n+1}=1$. Then, by Lemma \ref{oscillatory} with $s=s_1+s_2$ and $|x_0-x_1|, |x_{2n+2}-x_{2n+3}|\leq \tilde{\sigma}_{n+1}$, we obtain that 
\begin{equation}\label{oscillatory2}
\Big|\int_0^\infty m(\lambda)\chi_N(\sqrt{\lambda})\Im(e^{i\sqrt{\lambda}\tilde{\sigma}_{n+1}})d\lambda\Big|\lesssim\frac{N^2\|m\|_{\mathcal{H}(s_1+s_2)}}{\la N(x_0-x_1)\ra^{s_1}\la N(x_{2n+2}-x_{2n+3})\ra^{s_2}}.
\end{equation}
Applying \eqref{oscillatory2} to $\eqref{example1}$, we get the arbitrary polynomial decay away from $x_0=x_1$,
\begin{equation}
|\eqref{example1}|\lesssim \frac{N^2\|m\|_{\mathcal{H}(s_1+s_2)}k_{\eqref{example1}}^n(\tilde{x},\tilde{y})}{\la N(x_0-x_1)\ra^{s_1}\la N(x_{2n+2}-x_{2n+3})\ra^{s_2}}=\frac{N^2\|m\|_{\mathcal{H}(s_1+s_2)}k_{\eqref{example1}}^n(\tilde{x},\tilde{y})}{\la N(x-\tilde{x})\ra^{s_1}\la N(\tilde{y}-y)\ra^{s_2}},
\end{equation}
where
\begin{equation}
\begin{aligned}
k_{\eqref{example1}}^n(\tilde{x},\tilde{y}):&=\int_{\mathbb{R}^{6n}}\frac{\prod_{k=1}^{n+1}|\tilde{S}_0(x_{2k-1}, x_{2k})|\prod_{k=1}^n |V(x_{2k})|}{\prod_{k=1}^n 4\pi|x_{2k}-x_{2k+1}|}d\mathbf{x}_{(2,2n+1)}\\
&=[|\tilde{S}_0|(|V|(-\Delta)^{-1}|\tilde{S}_0|)^n](\tilde{x},\tilde{y})
\end{aligned}
\end{equation}
and $|\tilde{S}_0|$ is the integral operator with kernel $|\tilde{S}_0(x,y)|$. We claim that
\begin{equation}
\|k_{\eqref{example1}}^n(\tilde{x},\tilde{y})\|_{L_y^\infty L_{x_1}^1}\lesssim \tilde{S}^{n+1}(\|V\|_{\mathcal{K}}/4\pi)^n.
\end{equation}
Indeed, since $\||\tilde{S}_0|(|V|(-\Delta)^{-1}|\tilde{S}_0|)^n f\|_{L^1}\leq \tilde{S}^{n+1}(\frac{\|V\|_{\mathcal{K}}}{4\pi})^n\|f\|_{L^1}$ and $|\tilde{S}_0|(|V|(-\Delta)^{-1}|\tilde{S}_0|)^n$ is an integral operator, sending $f\to \delta(\cdot-y)$, we prove the claim.

Similarly, we write $\eqref{example2}$ as
\begin{equation}
\begin{aligned}
&\Im\int_0^\infty \int_{\mathbb{R}^{3n-3}} m(\lambda)\chi_N(\sqrt{\lambda})\prod_{k=1}^n V(x_{k})\frac{\prod_{k=0}^{n+1} e^{i{\alpha_k\sqrt{\lambda}}|x_k-x_{k+1}|}}{\prod_{k=1}^n 4\pi|x_{k}-x_{k+1}|}d\mathbf{x}_{(2,n)} d\lambda\\
&=\int_{\mathbb{R}^{3n-3}}\frac{\prod_{k=1}^n V(x_{k})}{\prod_{k=1}^n 4\pi|x_{k}-x_{k+1}|} \Big\{\int_0^\infty m(\lambda)\chi_N(\sqrt{\lambda}) \Im(e^{i\sqrt{\lambda}\tilde{\tilde{\sigma}}_{n+1}})d\lambda\Big\}d\mathbf{x}_{(2,n)}
\end{aligned}
\end{equation}
where $x_0:=x$, $x_1:=\tilde{x}$, $x_{n+1}:=\tilde{y}$, $x_{n+2}:=y$, $\alpha_0=\alpha_{n+2}=1$ and $\tilde{\tilde{\sigma}}_n:=\sum_{k=0}^n\alpha_k|x_k-x_{k+1}|$. Then, by Lemma \ref{oscillatory} with $s=s_1+s_2$ and $|x_0-x_1|, |x_{n+1}-x_{n+2}|\leq \tilde{\tilde{\sigma}}_{n+1}$, we obtain that 
\begin{equation}
|\eqref{example2}|\lesssim \frac{N^2\|m\|_{\mathcal{H}(s_1+s_2)}k_{\eqref{example2}}^n(\tilde{x},\tilde{y})}{\la N(x_0-x_1)\ra^{s_1}\la N(x_{n+1}-x_{n+2})\ra^{s_2}}= \frac{N^2\|m\|_{\mathcal{H}(s_1+s_2)}k_{\eqref{example2}}^n(\tilde{x},\tilde{y})}{\la N(x-\tilde{x})\ra^{s_1}\la N(\tilde{y}-y)\ra^{s_2}}
\end{equation}
where
\begin{equation}
k_{\eqref{example2}}^n(\tilde{x},\tilde{y}):=\int_{\mathbb{R}^{3n-3}}\frac{\prod_{k=1}^n |V(x_{k})|}{\prod_{k=1}^n 4\pi|x_{k}-x_{k+1}|}d\mathbf{x}_{(2,n)}=(4\pi)^{-n}(|V|(-\Delta)^{-1})^n(\tilde{x},\tilde{y}).
\end{equation}
Then by the definition of the global Kato norm, we prove that 
\begin{equation}
\begin{aligned}
\|k_{\eqref{example2}}^n(\tilde{x},\tilde{y})\|_{L_y^\infty L_{x_1}^1}\leq (\|V\|_{\mathcal{K}}/4\pi)^n.
\end{aligned}
\end{equation}
Similarly, we estimate other kernels, and define $k_1^n(\tilde{x},\tilde{y})$ as the sum of all $2^{2n+1}$ many upper bounds including $K_{\eqref{example1}}(\tilde{x},\tilde{y})$ and $K_{\eqref{example2}}(\tilde{x},\tilde{y})$. Then, $k_1^n(\tilde{x},\tilde{y})$ satisfies \eqref{eq1: summability in N: low} and \eqref{eq2: summability in N: low}.
\end{proof}

\begin{lemma}[Summability in $n$]\label{summability in n: low}
For any $\epsilon>0$, there exist a small number $N_0=N_0(V,\epsilon)\ll1$ and $k_2^n(\tilde{x},\tilde{y})\in L_{\tilde{y}}^\infty L_{\tilde{x}}^1$ such that for $N\leq N_0$,
\begin{equation}\label{eq1: summability in n: low}
|\textup{Pb}_N^n (x, \tilde{x}, \tilde{y}, y)|\lesssim N^2 \|m\|_{\mathcal{H}(s)}k_2^n(\tilde{x},\tilde{y})
\end{equation}
and
\begin{equation}\label{eq2: summability in n: low}
\|k_2^n(\tilde{x},\tilde{y})\|_{L_{\tilde{y}}^\infty L_{\tilde{x}}^1}\leq \epsilon^n.
\end{equation}
\end{lemma}

\begin{proof} 
Fix small $\epsilon>0$. Then, by Lemma \ref{VR estimate} $(ii)$, we choose small $N_0:=\delta=\delta(\epsilon)>0$ and an integral operator $B$ such that $|B_{\lambda,0}(x,y)|\leq B(x,y)$ for $0\leq\lambda\leq N_0$, and
\begin{equation}
\|B\|_{L^1\to L^1}\leq \epsilon(\tilde{S}+1)^{-1},
\end{equation}
where $\tilde{S}$ is a positive number given from \eqref{tildeS}. We define
$$k_2^n(\tilde{x},\tilde{y}):=[(I+|\tilde{S}_0|)(B(I+|\tilde{S}_0|))^n](\tilde{x},\tilde{y}),$$
where $|\tilde{S}_0|$ is the integral operator with $|\tilde{S}_0(x,y)|$ as kernel. Then, by definitions (see \eqref{PbNn: low}), one can check that $k_2^n(\tilde{x},\tilde{y})$ satisfies \eqref{eq1: summability in n: low}. For \eqref{eq2: summability in n: low}, splitting $(I+|\tilde{S}_0|)$ into $I$ and $|\tilde{S}_0|$ in $k_2^n(\tilde{x},\tilde{y})$, we get $2^{n+1}$ terms,
\begin{equation}
k_2^n(\tilde{x},\tilde{y})=[|\tilde{S}_0|(B|\tilde{S}_0|)^n](\tilde{x},\tilde{y})+\cdot\cdot\cdot+B^n(\tilde{x},\tilde{y}).
\end{equation}
For example, we consider $|\tilde{S}_0|(B|\tilde{S}_0|)^n$ and $B^n$. Since both $|\tilde{S}_0|$ and $B$ are integral operators, by Lemma 3.4 and $(5.12)$, we obtain
\begin{equation}
\begin{aligned}
\|[|\tilde{S}_0|(B|\tilde{S}_0|)^n](\tilde{x},\tilde{y})\|_{L_{\tilde{y}}^\infty L_{\tilde{x}}^1}&=\||\tilde{S}_0|(B|\tilde{S}_0|)^n\|_{L^1\to L^1}\leq \tilde{S}^{n+1}\Big(\epsilon(\tilde{S}+1)^{-1}\Big)^{n},\\
\|B^n(\tilde{x},\tilde{y})\|_{L_{\tilde{y}}^\infty L_{\tilde{x}}^1}&=\|B^n\|_{L^1\to L^1}\leq\Big(\epsilon(\tilde{S}+1)^{-1}\Big)^{n}.
\end{aligned}
\end{equation}
Similarly, we estimate other $2^{n+1}-2$ terms. Summing them up, we prove \eqref{eq2: summability in n: low}.
\end{proof}

\section{Medium Frequency Estimate: Proof of Lemma \ref{Key} $(iii)$}
The proof closely follows from that of Lemma \ref{Key} $(ii)$, so we only sketch the proof. For $\epsilon>0$, we take $\delta=\delta(\epsilon)>0$ from Lemma \ref{VR estimate} $(ii)$. We choose a partition of unity function $\psi\in C_c^\infty$ such that $\supp\psi\subset[-\delta, \delta]$, $\psi(\lambda)=1$ if $|\lambda|\leq \frac{\delta}{3}$ and $\sum_{j=1}^\infty\psi(\cdot-\lambda_j)\equiv1$ on $(0,+\infty)$, where $\lambda_j=j\delta$.

Let $N_0$ and $N_1$ be dyadic numbers chosen in the previous sections. For $N_0\leq N\leq N_1$, we first decompose $\chi_N(\sqrt{\lambda})$ in $\textup{Pb}_N$ (see \eqref{PbN}) into $\chi_N(\sqrt{\lambda})=\sum_{j=N/2\delta}^{2N/\delta}\chi_N^j(\lambda)$ where $\chi_N^j(\lambda)=\chi_N(\sqrt{\lambda})\psi(\lambda-\lambda_j)$. Plugging the formal series $(\ref{formal series1: low})$ with $\lambda_0=\lambda_j$ into each integral, we write the kernel of $\textup{Pb}_N$ as
\begin{equation}\label{formal series1: medium}
\textup{Pb}_N(x,y)``="\iint_{\mathbb{R}^6}\frac{V(\tilde{y})}{16\pi^3|x-\tilde{x}||\tilde{y}-y|}\Big[\sum_{n=0}^\infty (-1)^{n+1}\textup{Pb}_N^n(x,\tilde{x},\tilde{y}, y)\Big]d\tilde{x}d\tilde{y},
\end{equation}
where
\begin{equation}\label{PbNn: medium}
\begin{aligned}
&\textup{Pb}_N^n(x,\tilde{x},\tilde{y}, y)\\
&=\sum_{j=N/2\delta}^{2N/\delta}\int_0^\infty m(\lambda)\chi_N^j(\sqrt{\lambda})\Im[e^{i\sqrt{\lambda}(|x-\tilde{x}|+|\tilde{y}-y|)}[S_{\lambda_j}(B_{\lambda,\lambda_j}S_{\lambda_j})^n](\tilde{x},\tilde{y})]d\lambda.
\end{aligned}
\end{equation}
By the arguments in the previous sections, for Lemma \ref{Key} $(iii)$, it suffices to show the following two lemmas:
\begin{lemma}[Summability in $N$]\label{summability in N} For $N_0<N<N_1$, there exists $k_{N,1}^n(\tilde{x},\tilde{y})$ such that for $s_1, s_2\geq0$,
\begin{equation}\label{eq1: summability in N}
|\textup{Pb}_N^n (x, \tilde{x}, \tilde{y}, y)|\lesssim \frac{N^2\|m\|_{\mathcal{H}(s_1+s_2)}k_{N,1}^n(\tilde{x},\tilde{y})}{\la N(x-\tilde{x})\ra^{s_1}\la N(\tilde{y}-y)\ra^{s_2}},
\end{equation}
and
\begin{equation}\label{eq2: summability in N}
\|k_{N,1}^n(\tilde{x},\tilde{y})\|_{L_{\tilde{y}}^\infty L_{\tilde{x}}^1}\leq(\tilde{S}+1)^{n+1}\Big(\frac{\|V\|_{\mathcal{K}}}{2\pi}\Big)^n.
\end{equation}
\end{lemma}

\begin{proof}
For instance, consider
\begin{equation}\label{example: medium}
\int_0^\infty m(\lambda)\chi_N^j(\lambda)\Im[e^{i\sqrt{\lambda}(|x-\tilde{x}|+|\tilde{y}-y|)}\{S_{\lambda_j}(B_{\lambda,\lambda_j}S_{\lambda_j})^n\}(\tilde{x},\tilde{y})]d\lambda
\end{equation}
among $O(N)$-many similar integrals in \eqref{PbNn: medium}. As we did in Lemma \ref{oscillatory}, we show that 
\begin{equation}\label{oscillatory: medium}
\Big|\int_0^\infty m(\lambda)\chi_N^j(\lambda)\Im(e^{i\sqrt{\lambda}\sigma})d\lambda\Big|\lesssim_{N_0, N_1}\frac{N\|m\|_{\mathcal{H}(s)}}{ \la N\sigma\ra^s}.
\end{equation}
Repeating the proof of Lemma 5.1 (but replacing $S_0$ and $B_{\lambda,0}$ by $S_{\lambda_j}$ and $B_{\lambda,\lambda_j}$ and applying \eqref{oscillatory: medium} instead of Lemma \ref{oscillatory}), one can find $k_{N,j,1}^n(\tilde{x},\tilde{y})$ such that for $s_1, s_2\geq0$,
\begin{align}
|\eqref{example: medium}|&\lesssim \frac{N\|m\|_{\mathcal{H}(s_1+s_2)}k_{N,j,1}^n(\tilde{x},\tilde{y})}{\la N(x-\tilde{x})\ra^{s_1}\la N(\tilde{y}-y)\ra^{s_2}},\\
\|k_{N,j,1}^n(\tilde{x},\tilde{y})\|_{L_{\tilde{y}}^\infty L_{\tilde{x}}^1}&\leq(\tilde{S}+1)^{n+1}\Big(\frac{\|V\|_{\mathcal{K}}}{2\pi}\Big)^n.
\end{align}
Define
$$k_{N,1}^n(\tilde{x},\tilde{y}):=\frac{\delta}{N}\sum_{j=N/2\delta}^{2N/\delta}k_{N,j,1}^n(\tilde{x},\tilde{y}),$$
then it satisfies \eqref{eq1: summability in N} and \eqref{eq2: summability in N}.
\end{proof}

\begin{lemma}[Summability in $n$]
Let $\epsilon>0$ be a small number chosen at the beginning of this section. For $N_0<N<N_1$, there exists $k_{N,2}^n(\tilde{x},\tilde{y})$ such that
\begin{equation}
|\textup{Pb}_N^n (x, \tilde{x}, \tilde{y}, y)|\lesssim N^2 \|m\|_{\mathcal{H}(s)}k_{N,2}^n(\tilde{x},\tilde{y}),
\end{equation}
and
\begin{equation}
\|k_{N,2}^n(\tilde{x},\tilde{y})\|_{L_{\tilde{y}}^\infty L_{\tilde{x}}^1}\leq (1+\tilde{S})^{n+1}\epsilon^n.
\end{equation}
\end{lemma}

\begin{proof}
Again, we consider \eqref{example: medium}. By the choice of $\epsilon$ and $\delta$ and Lemma \ref{VR estimate} $(ii)$, there exists an integral operator $B$ such that $|B_{\lambda,\lambda_j}(x,y)|\leq B(x,y)$ for $|\lambda-\lambda_j|<\delta$, $\lambda,\lambda_j\geq0$, and $\|B\|_{L^1\to L^1}\leq \epsilon$. Let $|\tilde{S}_{\lambda_j}|$ be the integral operator with integral kernel $|\tilde{S}_{\lambda_j}(x,y)|$. Then, we have 
\begin{equation}
|(6.5)|\lesssim N\|m\|_{\mathcal{H}(s)}[(I+|\tilde{S}_{\lambda_j}|)(B(I+|\tilde{S}_{\lambda_j}|))^n](\tilde{x},\tilde{y})
\end{equation}
and
\begin{equation}
\|[(I+|\tilde{S}_{\lambda_j}|)(B(I+|\tilde{S}_{\lambda_j}|))^n](\tilde{x},\tilde{y})\|_{L_{\tilde{y}}^\infty L_{\tilde{x}}^1}\leq (1+\tilde{S})^{n+1}\epsilon^n.
\end{equation}
Therefore, we define
\begin{equation}
k_2^n(\tilde{x},\tilde{y}):=\frac{\delta}{N}\sum_{j=N/2\delta}^{2N/\delta}[(I+|\tilde{S}_{\lambda_j}|)(B(I+|\tilde{S}_{\lambda_j}|))^n](\tilde{x},\tilde{y}),
\end{equation}
then it satisfies $(6.9)$ and $(6.10)$.
\end{proof}

\section{Application to the Nonlinear Schr\"odinger Equation}

\subsection{Two norm estimates}
Following the argument in \cite{DFVV}, we begin with proving the boundedness of the imaginary power operators. For $\alpha\in\mathbb{R}$, the imaginary power operator $H^{i\alpha}P_c$ is defined as a spectral multiplier of symbol $\lambda^{i\alpha}1_{[0,+\infty)}$. We consider $H^{i\alpha}P_c$ instead of $H^{i\alpha}$ just for convenience's sake. Indeed, by the assumptions, $H$ has only finitely many negative eigenvalues, and the projection $P_{\lambda_j}$ is bounded on $L^r$ for any $1<r<\infty$ (see  Lemma \ref{eigenfunctions}). Therefore, the boundedness of $H^{i\alpha}P_c$ implies that of $H^{i\alpha}=H^{i\alpha}P_c+\sum \lambda_j^{i\alpha}P_{\lambda_j}$, where $\lambda_j$'s are negative eigenvalues of $H$.

\begin{lemma}[Imaginary power operator]\label{lem:ImaginaryPower} If $V\in\mathcal{K}_0\cap L^{3/2,\infty}$ and $H$ has no eigenvalue or resonance on $[0,+\infty)$, then for $\alpha\in\mathbb{R}$,
\begin{equation}
\|H^{i\alpha}P_c\|_{L^r\to L^r}\lesssim \la\alpha\ra^3,\ 1<r<\infty.
\end{equation}
\end{lemma}

\begin{proof}
Since $\|\lambda^{i\alpha}1_{[0,+\infty)}\|_{\mathcal{H}(3)}\lesssim\la\alpha\ra^3$, the lemma follows from Theorem \ref{main theorem}.
\end{proof}

\begin{proposition}[Two norm estimates]\label{cor:Riesz}If $V\in\mathcal{K}_0\cap L^{3/2,\infty}$ and $H$ has no eigenvalue or resonance on $[0,+\infty)$, then for $0\leq s\leq 2$ and $1<r<\frac{3}{s}$,
\begin{align}
\|H^\frac{s}{2}P_c(-\Delta)^{-\frac{s}{2}}f\|_{L^r}&\lesssim\|f\|_{L^r},\label{two norm estimates1} \\
\|(-\Delta)^{\frac{s}{2}}H^{-\frac{s}{2}}P_cf\|_{L^r}&\lesssim\|f\|_{L^r}.\label{two norm estimates2}
\end{align}
\end{proposition}

\begin{proof}
\eqref{two norm estimates1}: Pick $f,g\in L^1\cap L^\infty$ such that $\supp\hat{f}\subset B(0,R)\setminus B(0,r)$, $P_{n\leq\cdot\leq N}g=P_cg$ for some $R, r, N, n>0$. Note that by Lemma 3.5, the collection of such $f$ ($g$, resp) is dense in $L^r$ ($L^{r'}$, resp). We define 
\begin{equation}
F(z):=\la H^zP_c(-\Delta)^{-z}f, g\ra_{L^2}=\la (-\Delta)^{-\Re z-i\Im z}f, H^{-i\Im z}H^{\Re z}g\ra_{L^2}.
\end{equation}
Indeed, $F(z)$ is well-defined, since $(-\Delta)^{-\Re z-i\Im z}f, H^{-i\Im z}H^{\Re z}g\in L^2$. Moreover, $F(z)$ is continuous on $S=\{z: 0\leq\Re z\leq 1\}\subset\mathbb{C}$, and it is analytic in the interior of $S$. We claim that $HP_c(-\Delta)^{-1}$ is bounded on $L^r$ for $1<r<\frac{3}{2}$. Indeed, by Lemma 3.6 $(i)$, 
\begin{equation}
\|HP_c(-\Delta)^{-1}f\|_{L^r}\lesssim\|(-\Delta+V)(-\Delta)^{-1}f\|_{L^r}\leq \|f\|_{L^r}+\|V(-\Delta)^{-1}f\|_{L^r}.
\end{equation}
By the H\"older inequality (Lemma A.2) and the Sobolev inequality in the Lorentz norms (Corollary A.6), we have
\begin{equation}
\|V(-\Delta)^{-1}f\|_{L^r}\leq \|V\|_{L^{3/2,\infty}}\|(-\Delta)^{-1}f\|_{L^{\frac{3r}{3-2r},r}}\lesssim\|f\|_{L^r}.
\end{equation}
Hence, by the claim and Lemma \ref{lem:ImaginaryPower}, we get
\begin{align}
|F(1+i\alpha)|&\leq\|H^{1+i\alpha}P_c (-\Delta)^{-1-i\alpha}f\|_{L^r}\|g\|_{L^{r'}}\lesssim \la\alpha\ra^6\|f\|_{L^r}\|g\|_{L^{r'}},&&(1<r<\tfrac{3}{2}),\\
|F(i\alpha)|&\leq\|H^{i\alpha}P_c(-\Delta)^{-i\alpha}f\|_{L^r}\|g\|_{L^{r'}}\lesssim \la\alpha\ra^6\|f\|_{L^r}\|g\|_{L^{r'}},&&(1<r<\infty).
\end{align}
Therefore \eqref{two norm estimates1} follows from the Stein's complex interpolation theorem.\\
\eqref{two norm estimates2}: Pick $f$ and $g$ as above, and consider
\begin{equation}
G(z):=\la (-\Delta)^zH^{-z}P_c g, f\ra_{L^2}.
\end{equation}
We claim that $(-\Delta)H^{-1}P_cg$ is bounded on $L^r$ for $1<r<\frac{3}{2}$. By the triangle inequality,
\begin{equation}
\|(-\Delta)H^{-1}P_cg\|_{L^r}=\|(H-V)H^{-1}P_cg\|_{L^r}\leq\|P_cg\|_{L^r}+\|VH^{-1}P_cg\|_{L^r}.
\end{equation}
By Lemma 3.6 $(i)$, $\|P_cg\|_{L^r}\lesssim\|g\|_{L^r}$. By the H\"older inequality in the Lorentz norms (Lemma A.2) and the Sobolev inequality associated with $H$ \cite[Theorem 1.9]{H1}, we get
\begin{equation}
\|VH^{-1}P_cg\|_{L^r}\leq \|V\|_{L^{3/2,\infty}}\|H^{-1}P_cg\|_{L^{\frac{3r}{3-2r},r}}\lesssim \|V\|_{L^{3/2,\infty}}\|g\|_{L^r}.
\end{equation}
Repeating the above argument with the complex interpolation, we complete the proof.
\end{proof}

\subsection{Local well-posedness}
Now we are ready to show the local well-posedness (LWP) of a 3d quintic nonlinear Schr\"odinger equation
\begin{equation}\tag{$\textup{NLS}_V$}
iu_t+\Delta u-Vu\pm|u|^4u=0;\ u(0)=u_0.
\end{equation}

\begin{theorem}[LWP]  If $V\in\mathcal{K}_0\cap L^{3/2,\infty}$ and $H$ has no eigenvalue or resonance on $[0,+\infty)$, then $(\textup{NLS}_V)$ is locally well-posed in $\dot{H}^1$. Precisely, for $A>0$, there exists $\delta=\delta(A)>0$ such that for an initial data $u_0\in\dot{H}^1$ obeying
\begin{equation}
\|\nabla u_0\|_{L^2}\leq A\textup{ and }\|e^{-itH} u_0\|_{L_{t\in [0,T_0]}^{10}L_x^{10}}<\delta,
\end{equation}
$(\textup{NLS}_V)$ has a unique solution $u\in C_t(I; \dot{H}_x^1)$, with $I=[0,T)\subset [0,T_0]$, such that
\begin{equation}
\|\nabla u\|_{L_{t\in I}^{10}L_x^{30/13}}<\infty\textup{ and }\|u\|_{L_{t\in I}^{10}L_x^{10}}<2\delta.
\end{equation}
\end{theorem}

\begin{proof}
\textit{(Step 1. Contraction mapping argument)}
Let $\psi_j$ be the eigenfunction corresponding to the negative eigenvalue $\lambda_j$ normalized so that $\|\psi_j\|_{L^2}=1$. Choose small $T\in(0,T_0)$ such that $\|\psi_j\|_{L_{t\in I}^{10}L_x^{10}}, \|\psi_j\|_{L_{t\in I}^2L_x^2}\leq 1$ for all $j$, where $I=[0,T]$ and $\psi_j(t,x)=\psi_j(x)$ for all $t\in I$. For notational convenience, we omit the time interval $I$ in the norm $\|\cdot\|_{L_{t\in I}^p}$ if there is no confusion. Following a standard contraction mapping argument \cite{Caz, Tao1}, we aim to show that
\begin{equation}
\Phi_{u_0}(v)(t):=e^{-itH}u_0\pm i\int_0^t e^{-i(t-s)H}(|v|^4v(s))ds
\end{equation}
is a contraction map on the set
\begin{equation}
B_{a,b}:=\{v: \|v\|_{L_{t,x}^{10}}\leq a,\ \|\nabla v\|_{L_t^{10}L_x^{30/13}}\leq b\},
\end{equation}
equipped with the metric $d(u,v)=\|u-v\|_{L_{t,x}^{10}}+\|\nabla(u-v)\|_{L_t^{10}L_x^{30/13}}$, where $a,b$ and $\delta$ will be chosen later.

We claim that $\Phi_{u_0}$ maps from $B_{a,b}$ to itself. We write
\begin{equation}
\begin{aligned}
\|\Phi_{u_0}(v)\|_{L_{t,x}^{10}}&\leq\|e^{-itH}u_0\|_{L_{t,x}^{10}}+\Big\|\int_0^t e^{-i(t-s)H}P_c(|v|^4v(s))ds\Big\|_{L_{t,x}^{10}}\\
&\quad\quad+\sum_{j=1}^J\Big\|\int_0^t e^{-i(t-s)H}(\la |v|^4v(s),\psi_j\ra_{L^2}\psi_j)ds\Big\|_{L_{t,x}^{10}}\\
&=I+II+\sum_{j=1}^JIII_j.
\end{aligned}
\end{equation}
By assumption, $I\leq\delta$. For $II$, by the Sobolev inequality associated with $H$ \cite[Theorem 1.6]{H1}, Strichartz estimates (Proposition 1.2) and the two norm estimates, we get
\begin{align}
II&\lesssim\Big\|\int_0^t e^{-i(t-s)H}P_cH^{1/2}(|v|^4v(s))ds\Big\|_{L_t^{10}L_x^{30/13}}\lesssim\|H^{1/2}P_c(|v|^4v)\|_{L_t^2L_x^{6/5}}\nonumber\\
&\lesssim\|\nabla (|v|^4v)\|_{L_t^2L_x^{6/5}}\leq3\|(v^2\nabla v)(\bar{v})^2\|_{L_t^2L_x^{6/5}}+2\|(v^2\nabla v)(\bar{v})^2\|_{L_t^2L_x^{6/5}}\\
&\lesssim\|v\|_{L_{t,x}^{10}}^4\|\nabla v\|_{L_t^{10}L_x^{30/13}}\leq a^4b.\nonumber
\end{align}
For the last term, by the H\"older inequality, the choice of $T$ and $(7.17)$, we obtain
\begin{equation}
\begin{aligned}
III_j&=\Big\|\int_0^t e^{-i(t-s)\lambda_j}(\la |v|^4v(s),\psi_j\ra_{L^2}\psi_j)ds\Big\|_{L_{t,x}^{10}}\\
&\leq\Big(\int_0^T |\la |v|^4v(s),\psi_j\ra_{L^2}|ds\Big)\|\psi_j\|_{L_{t,x}^{10}}\\
&\leq \|\nabla(|v|^4v)\|_{L_t^2 L_x^{6/5}}\||\nabla|^{-1}\psi_j\|_{L_t^2 L_x^6}\\
&\lesssim\|\nabla(|v|^4v)\|_{L_t^2 L_x^{6/5}}\|\psi_j\|_{L_t^2 L_x^2}\leq a^4b.
\end{aligned}
\end{equation}
Therefore, we prove that 
\begin{equation}
\|\Phi_{u_0}(v)\|_{L_{t,x}^{10}}\leq\delta+Ca^4b.
\end{equation}
Next, we write
\begin{equation}
\begin{aligned}
\|\nabla\Phi_{u_0}(v)\|_{L_t^{10}L_x^{30/13}}&\leq\|\nabla P_c\Phi_{u_0}(v)\|_{L_t^{10}L_x^{30/13}}+\sum_{j=1}^J\|\nabla P_{\lambda_j}\Phi_{u_0}(v)\|_{L_t^{10}L_x^{30/13}}\\
&=\tilde{I}+\sum_{j=1}^J\tilde{II}_j.
\end{aligned}
\end{equation}
For $\tilde{I}$, by the two norm estimates, Strichartz estimates and $(7.17)$, we obtain
\begin{equation}
\begin{aligned}
\tilde{I}&\lesssim\|H^{1/2}P_c\Phi_{u_0}(v)\|_{L_t^{10}L_x^{30/13}}\\
&\lesssim\|H^{1/2}P_cu_0\|_{L^2}+\|H^{1/2}P_c(|v|^4v)\|_{L_t^2L_x^{6/5}}\\
&\lesssim\|\nabla u_0\|_{L^2}+\|H^{1/2}P_c(|v|^4v)\|_{L_t^2L_x^{6/5}}\lesssim A+a^4b.
\end{aligned}
\end{equation}
For $\tilde{II}$, by the H\"older inequality, $(7.19)$ and Lemma \ref{eigenfunctions}, we get
\begin{equation}
\begin{aligned}
\tilde{II}_j&\leq\|\la\Phi_{u_0}(v),\psi_j\ra_{L^2}\|_{L_t^{10}}\|\psi_j\|_{L^{30/13}}\\
&\leq\|\Phi_{u_0}(v)\|_{L_{t,x}^{10}}\|\psi_j\|_{L_x^{10/9}}\lesssim \delta+a^4b.
\end{aligned}
\end{equation}
Collecting all, we prove that
\begin{equation}
\|\nabla\Phi_{u_0}(v)\|_{L_t^{10}L_x^{30/13}}\leq CA+Ca^4b.
\end{equation}
Let $b=2AC$, $a=\min((2C)^{-\frac{1}{4}}, (2Cb)^{-\frac{1}{3}})$ and $\delta=\frac{a}{2}$ $(\Rightarrow Ca^4b\leq AC$ and $Ca^3b\leq\frac{1}{2})$. Then, by $(7.19)$ and $(7.23)$, $\Phi_{u_0}$ maps from $B_{a,b}$ to itself. Similarly, one can show that $\Phi_{u_0}$ is contractive in $B_{a,b}$. Thus, we conclude that there exists unique $u\in B_{a,b}$ such that
\begin{equation}
u(t)=\Phi_{u_0}(u)=e^{-itH}u_0+i\int_0^t e^{-i(t-s)H} (|u|^4u)(s)ds.
\end{equation}
\textit{(Step 2. Continuity)} In order to show that $u(t)\in C_t(I; \dot{H}_x^1)$, we write 
\begin{equation}
\begin{aligned}
u(t)&=e^{-itH}(P_cu_0+\sum_{j=1}^JP_{\lambda_j}u_0)\\
&\quad\pm i\int_0^t e^{-i(t-s)H} \Big(P_c(|u|^4u)(s)+\sum_{j=1}^J P_{\lambda_j}(|u|^4u)(s)\Big) ds\\
&=e^{-itH}P_cu_0+\sum_{j=1}^J e^{-it\lambda_j}P_{\lambda_j}u_0\pm i\int_0^t e^{-i(t-s)H} P_c(|u|^4u)(s)ds\\
&\ \quad\quad\quad\quad\quad\pm i\sum_{j=1}^J \int_0^t e^{-i(t-s)\lambda_j}P_{\lambda_j}(|u|^4u)(s) ds\\
&=:I(t)+\sum_{j=1}^JII_j(t)+III(t)+\sum_{j=1}^JIV_j(t).
\end{aligned}
\end{equation}
For $I(t)$, by the two norm estimates and $L^2$-continuity of $e^{-itH}$, we have
\begin{equation}
\begin{aligned}
\|I(t)-I(t_0)\|_{\dot{H}^1}\lesssim\|(e^{-itH}-e^{-it_0H})H^{1/2}P_cu_0\|_{L^2}\to0\quad\textup{as }t\to t_0,
\end{aligned}
\end{equation}
since $\|H^{1/2}P_cu_0\|_{L^2}\lesssim\|u\|_{\dot{H}^1}<\infty$. $II_j(t)$ is continuous in $\dot{H}^1$, since 
\begin{equation}
\|P_{\lambda_j}u_0\|_{\dot{H}^1}=|\la u_0,\psi_j\ra_{L^2}|\|\psi_j\|_{\dot{H}^1}\lesssim\|u_0\|_{\dot{H}^1}\|\psi_j\|_{\dot{H}^{-1}}\lesssim\|u_0\|_{\dot{H}^1}\|\psi_j\|_{L^{6/5}}<\infty.
\end{equation}
For $III(t)$, by the two norm estimates, Strichartz estimates and $(7.17)$, we have
\begin{equation}
\begin{aligned}
\|III(t)-III(t_0)\|_{\dot{H}^1}&\lesssim\|H^{1/2}(III(t)-III(t_0))\|_{L^2}\\
&\lesssim\|H^{1/2}P_c(|u|^4u)\|_{L_{s\in [t_0, t]}^2 L_x^{6/5}}\to 0\quad\textup{as }t\to t_0.
\end{aligned}
\end{equation}
For $IV_j(t)$, by the H\"older inequality and $(7.17)$, we write
\begin{equation}
\begin{aligned}
\|IV_j(t)-IV_j(t_0)\|_{\dot{H}^1}&\leq \|\psi_j\|_{\dot{H}^1}\|\nabla(|u|^4u)(s)\|_{L_{t\in [t_0,t]}^2L_x^{6/5}}\||\nabla|^{-1}\psi_j\|_{L_{s\in [t_0,t]}^2L_x^6}\\
&\lesssim\|\nabla(|u|^4u)(s)\|_{L_{s\in [t_0,t]}^2L_x^{6/5}}\|\psi_j\|_{L_t^2L_x^2}\to0\quad\textup{as }t\to t_0.
\end{aligned}
\end{equation}
Collecting all, we conclude that $u(t)$ is continuous in $\dot{H}^1$.
\end{proof}

\appendix

\section{Lorentz Spaces and Interpolation Theorem}

Following \cite{Tao1}, we summarize useful properties of the Lorentz spaces. Let $(X,\mu)$ be a measure space. The Lorentz (quasi) norm is defined by
\begin{equation}
\|f\|_{\tilde{L}^{p,q}}:=\left\{\begin{aligned}
&p^{1/q} \|\lambda\mu(\{|f|\geq\lambda\})^{1/p}\|_{L^q((0,+\infty),\tfrac{d\lambda}{\lambda})}&&\textup{ when }1\leq p<\infty\textup{ and }1\leq q\leq\infty;\\
&\|f\|_{L^\infty}&&\textup{ when }p=q=\infty.
\end{aligned}
\right.
\end{equation}

\begin{lemma}[Properties of the Lorentz spaces] Let $1\leq p\leq\infty$ and $1\leq q, q_1, q_2\leq\infty$.\\
$(i)$ $L^{p,p}=L^p$, and $L^{p,\infty}$ is the weak $L^p$-space.\\
$(ii)$ If $q_1\leq q_2$, $L^{p,q_1}\subset L^{p,q_2}$.
\end{lemma}

\begin{lemma}[H\"older inequality] If $1\leq p, p_1, p_2, q,q_1,q_2\leq\infty$, $\frac{1}{p}=\frac{1}{p_1}+\frac{1}{p_2}$ and $\frac{1}{q}=\frac{1}{q_1}+\frac{1}{q_2}$, then
\begin{equation}
\|fg\|_{L^{p,q}}\lesssim\|f\|_{L^{p_1,q_1}}\|fg\|_{L^{p_2,q_2}}.
\end{equation}
\end{lemma}

\begin{lemma}[Dual characterization of $L^{p,q}$] If $1<p<\infty$ and $1\leq q\leq\infty$, then
\begin{equation}
\|f\|_{L^{p,q}}\sim\sup_{\|g\|_{L^{p',q'}}\leq 1}\Big|\int_X f\bar{g} d\mu\Big|.
\end{equation}
\end{lemma}

A measurable function $f$ is called a \textit{sub-step function} of height $H$ and width $W$ if $f$ is supported on a set $E$ with measure $\mu(E)=W$ and $|f(x)|\leq H$ almost everywhere. Let $T$ be a linear operator that maps the functions on a measure space $(X,\mu_X)$ to functions on another measure space $(Y,\mu_Y)$. We say that $T$ is \textit{restricted weak-type} $(p,\tilde{p})$ if 
\begin{equation}
\|Tf\|_{L^{\tilde{p},\infty}}\lesssim HW^{1/p}
\end{equation}
for all sub-step functions $f$ of height $H$ and width $W$.

\begin{theorem}[Marcinkiewicz interpolation theorem]
Let $T$ be a linear operator such that 
\begin{equation}
\la Tf,g\ra_{L^2}=\int_Y Tf \bar{g}d\mu_Y
\end{equation}
is well-defined for all simple functions $f$ and $g$. Let $1\leq p_0,p_1, \tilde{p}_0, \tilde{p}_1\leq\infty$. Suppose that $T$ is restricted weak-type $(p_i, \tilde{p}_i)$ with constant $A_i>0$ for $i=0,1$. Then,
\begin{equation}
\|Tf\|_{L^{\tilde{p}_\theta,q}}\lesssim A_0^{1-\theta}A_1^\theta \|f\|_{L^{p_\theta,q}},\end{equation}
where $0<\theta<1$, $\frac{1}{p_\theta}=\frac{1-\theta}{p_0}+\frac{\theta}{p_1}$, $\frac{1}{\tilde{p}_\theta}=\frac{1-\theta}{\tilde{p}_0}+\frac{\theta}{\tilde{p}_1}$, $\tilde{p}_\theta>1$ and $1\leq q\leq\infty$.
\end{theorem}

In this paper, we use the interpolation theorem of the following form.
\begin{corollary}[Marcinkiewicz interpolation theorem]\label{Interpolation} Let $T$ be a linear operator. Let $1\leq p_1<p_2\leq\infty$. Suppose that for $i=0,1$, $T$ is bounded from $L^{p_i,1}$ to $L^{p_i,\infty}$. Then $T$ is bounded on $L^p$ for $p_1<p<p_2$.
\end{corollary}

\begin{proof}
The corollary follows from Theorem A.4, since $T$ is restricted weak-type $(p_i, p_i)$:
\begin{equation}
\|f\|_{L^{p_i,1}}=p_i\int_0^\infty \mu(|f|\geq\lambda)^{1/p_i} d\lambda\leq p_i\int_0^H W^{1/p_i}d\lambda=p_i HW^{1/p_i},
\end{equation}
for a sub-step function $f$ of height $H$ and width $W$.
\end{proof}

\begin{corollary}[Fractional integration inequality in the Lorentz spaces]\label{Fractional}
\begin{equation}
\Big\|\int_{\mathbb{R}^d}\frac{f(y)}{|x-y|^{d-s}}dy\Big\|_{L^{q,r}(\mathbb{R}^d)}\lesssim\|f\|_{L^{p,r}},
\end{equation}
where $1<p<q<\infty$, $1\leq r\leq\infty$ and $\frac{1}{q}=\frac{1}{p}-\frac{s}{d}$. At the endpoints, we have
\begin{align}
\Big\|\int_{\mathbb{R}^d}\frac{f(y)}{|x-y|^{d-s}}dy\Big\|_{L^{\frac{d}{d-s},\infty}(\mathbb{R}^d)}&\lesssim\|f\|_{L^1}, \Big\|\int_{\mathbb{R}^d}\frac{f(y)}{|x-y|^{d-s}}dy\Big\|_{L^\infty(\mathbb{R}^d)}\lesssim\|f\|_{L^{d/s,1}}.
\end{align}
\end{corollary}

\begin{proof}
$(A.9)$ follows from \cite[Theorem 1, p.119]{Stein} and duality. Then, $(A.8)$ follows from Corollary A.5.
\end{proof}


\begin{thebibliography}{00}

\bibitem{B} M. Beceanu, Structure of wave operators for a scaling-critical class of potentials. Amer. J. Math. 136 (2014), no. 2, 255-308.

\bibitem{BG} M. Beceanu and M. Goldberg, Schr\"odinger dispersive estimates for a scaling-critical class of potentials, Comm. Math. Phys. 314 (2012), no. 2, 471-481.

\bibitem{Caz} T. Cazenave, Semilinear Schr\"odinger equations. Courant Lecture Notes in Mathematics, 10. New York University, Courant Institute of Mathematical Sciences, New York; American Mathematical Society, Providence, RI, 2003. xiv+323 pp. 

\bibitem{C} M. Christ, $L^p$ bounds for spectral multipliers on nilpotent groups. Trans. Amer. Math. Soc. 328 (1991), no. 1, 73-81. 

\bibitem{DOS} X. T. Duong, E. M. Ouhabaz and A. Sikora, Plancherel-type estimates and sharp spectral multipliers. J. Funct. Anal. 196 (2002), no. 2, 443-485. 

\bibitem{DFVV} P. D'Ancona, L. Fanelli, L. Vega and N. Visciglia, Endpoint Strichartz estimates for the magnetic Schr\"odinger equation. J. Funct. Anal. 258 (2010), no. 10, 3227-3240.

\bibitem{DP} P. D'Ancona and V. Pierfelice, On the wave equation with a large rough potential. J. Funct. Anal. 227 (2005), no. 1, 30-77.

\bibitem{G1} M. Goldberg, Dispersive estimates for the three-dimensional Schr\"odinger equation with rough potentials. Amer. J. Math. 128 (2006), no. 3, 731-750.

\bibitem{G2} M. Goldberg, Dispersive bounds for the three-dimensional Schr\"odinger equation with almost critical potentials. Geom. Funct. Anal. 16 (2006), no. 3, 517-536.

\bibitem{GS1} M. Goldberg and W. Schlag, Dispersive estimates for Schr\"odinger operators in dimensions one and three. Comm. Math. Phys. 251 (2004), no. 1, 157-178.

\bibitem{GS2} M. Goldberg and W. Schlag, A limiting absorption principle for the three-dimensional Schr\"odinger equation with $L^p$ potentials. Int. Math. Res. Not. 2004, no. 75, 4049-4071.

\bibitem{H1} Y. Hong, A Remark on the Littlewood-Paley Projection, arXiv.org/abs/1206.4462.

\bibitem{H2} Y. Hong, Local-in-time well-posedness for nonlinear Schr\"odinger equations with potentials, expository note available at arXiv:1301.0402.

\bibitem{Hor} L. H\"ormander, Estimates for translation invariant operators in $L^p$ spaces. Acta Math. 104 (1960) 93-140. 

\bibitem{JSS} J.-L. Journe, A. Soffer, and C. Sogge, Decay estimates for Schr\"odinger operators. Comm. Pure Appl. Math. 44 (1991), no. 5, 573-604.

\bibitem{KT} M. Keel and T. Tao, Endpoint Strichartz estimates. Amer. J. Math. 120 (1998), no. 5, 955-980. 

\bibitem{MM} G. Mauceri and S. Meda, Vector-valued multipliers on stratified groups, Rev. Mat. Iberoamericana 6 (3-4) (1990) 141-154.

\bibitem{RSch} I. Rodnianski and W. Schlag, Time decay for solutions of Schr\"odinger equations with rough and time-dependent potentials. Invent. Math. 155 (2004), no. 3, 451-513.

\bibitem{Stein} E. Stein, Singular integrals and differentiability properties of functions. Princeton Mathematical Series, No. 30 Princeton University Press, Princeton, N.J. 1970 xiv+290 pp. 

\bibitem{Sh} Z. Shen, $L^p$ estimates for Schr\"odinger operators with certain potentials. Ann. Inst. Fourier (Grenoble) 45 (1995), no. 2, 513-546.

\bibitem{T} M. Takeda, Gaussian bounds of heat kernels for Schr\"odinger operators on Riemannian manifolds. Bull. Lond. Math. Soc. 39 (2007), no. 1, 85-94.

\bibitem{Tao1} T. Tao, Nonlinear dispersive equations. Local and global analysis. CBMS Regional Conference Series in Mathematics, 106; American Mathematical Society, Providence, RI, 2006. xvi+373 pp.

\bibitem{Tao2} T. Tao, Lecture notes for Math 247A : Fourier analysis, http://www.math.ucla.edu/$\sim$tao/247a.1.06f/

\bibitem{Y} K. Yajima,  The $W^{k,p}$-continuity of wave operators for Schr\"odinger operators. J. Math. Soc. Japan 47 (1995), no. 3, 551-581.

\end{thebibliography}
\end{document}